\documentclass[a4paper,reqno,11pt]{amsart} 

\usepackage[T1]{fontenc}
\usepackage[utf8]{inputenc}
\usepackage{xspace,
		amsfonts}
\usepackage{amsmath} 
\usepackage{amssymb}
\usepackage{graphicx,
		bbm,
		tikz,
		subfig}
  \usepackage{xfrac}
		\usepackage{dsfont}
\usepackage{amsthm}
\usepackage{mathtools}
\usepackage{enumitem}
\usepackage{verbatim}
\usepackage[autostyle]{csquotes}

\usepackage{cite}

\mathtoolsset{showonlyrefs}

\usepackage{geometry}
\newtheorem{theorem}{Theorem}[section]
\newtheorem{lemma}[theorem]{Lemma}
\newtheorem{proposition}[theorem]{Proposition}

\newtheorem{remark}[theorem]{Remark}

\newcommand{\R}{\mathbb{R}}

\newcommand{\G}{\mathcal{G}}

\newcommand{\N}{\mathbb{N}}
\newcommand{\E}{\mathbb{E}}

\newcommand{\V}{\mathbb{V}}

\newcommand{\HH}{\mathcal{H}}
\newcommand{\K}{\mathcal{K}}

\newcommand{\Z}{\mathbb{Z}}

\newcommand{\NN}{\mathcal N}
\newcommand{\LL}{\mathcal L}
\newcommand{\JJ}{\mathcal J}
\newcommand{\EE}{\mathcal{E}}

\newcommand\vv{\textsc{v}}

\tikzstyle{nodino}=[circle,draw,fill,inner sep=0pt,minimum size=0.5mm]
\tikzstyle{infinito}=[circle,inner sep=0pt,minimum size=0mm]
\tikzstyle{nodo}=[circle,draw,fill,inner sep=0pt, minimum size=0.5*width("k")]
\tikzstyle{nodo_vuoto}=[circle,draw,inner sep=0pt, minimum size=0.5*width("k")]
\usetikzlibrary{graphs}
\tikzset{every loop/.style={min distance=10mm,in=300,out=240,looseness=10}}
\tikzset{place/.style={circle,thick,draw=blue!75,fill=blue!20,minimum
		size=6mm}}
\tikzset{place2/.style={circle,thick,draw=red!75,fill=red!20,minimum
		size=6mm}}

\title[ ]{Non--uniqueness of normalized ground states for nonlinear Schr\"odinger equations on metric graphs}
\author[ ]{Simone Dovetta}
\address[Simone Dovetta]{Politecnico di Torino, Dipartimento di Scienze Matematiche ``G.L. Lagrange'', Corso Duca degli Abruzzi, 24, 10129 Torino, Italy}
\email{simone.dovetta@polito.it}

\begin{document}

\maketitle

\begin{abstract}
	We establish general non-uniqueness results for normalized ground states of nonlinear Schr\"odinger equations with power nonlinearity on metric graphs. Basically, we show that, whenever in the $L^2-$subcritical regime a graph hosts ground states at every mass, for nonlinearity powers close to the $L^2$-critical exponent $p=6$ there is at least one value of the mass for which ground states are non-unique. As a consequence, we also show that, for all such graphs and nonlinearities, there exist action ground states that are not normalized ground states.
\end{abstract}

\section{Introduction and main results}
The aim of the present paper is to provide general non-uniqueness results for normalized ground states of NonLinear Schr\"odinger (NLS) equations on metric graphs.
\smallskip

Given a {\em connected metric graph} $\G=(\V_\G,\E_\G)$, i.e. a one-dimensional singular variety obtained gluing together a certain number of intervals (the {\em edges}) at some of their endpoints (the {\em vertices}), an {\em energy ground state} with mass $\mu$ (or {\em normalized} ground state with mass $\mu$) is a solution $u\in H_\mu^1(\G)$ of the minimization problem
\begin{equation}
	\label{eq:infE}
	\EE_{p,\G}(\mu):=\inf_{u\in H_\mu^1(\G)}E_p(u,\G)
\end{equation}
where, for fixed $p>2$, the NLS energy functional $E_p: H^1(\G)\to\R$ is defined by
\begin{equation}
	\label{energy}
	E_p(u,\G):=\frac12\|u'\|_{L^2(\G)}^2-\frac1p\|u\|_{L^p(\G)}^p\,,
\end{equation}
and $H_\mu^1(\G)$ is the set of functions with prescribed mass (the $L^2$-norm) equal to $\mu$
\[
H_\mu^1(\G):=\left\{u\in H^1(\G)\,:\, \|u\|_{L^2(\G)}^2=\mu\right\}.
\]
It is well-known that, up to a change of sign, any ground state $u$ is a positive solution of the NLS equation with homogeneous Kirchhoff conditions at the vertices
\begin{equation}
	\label{eq:nlse}
	\begin{cases}
	u''+|u|^{p-2}u=\lambda u & \text{on each edge }e\in\E_\G\\
	\sum_{e\succ \vv}\frac{du}{dx_e}(\vv)=0 & \text{for every vertex }\vv\in\V_\G
	\end{cases}
\end{equation}
with 
\begin{equation}
\label{eq:Lu}
\lambda=\LL_p(u,\G):=\frac{\|u\|_{L^p(\G)}^p-\|u'\|_{L^2(\G)}^2}{\|u\|_{L^2(\G)}^2}\,.
\end{equation}
Here, $e\succ \vv$ means that the edge $e$ is incident at the vertex $\vv$, and $\displaystyle\frac{du}{dx_e}(\vv)$ stands for the outgoing derivative of $u$ at the vertex $\vv$ along the edge $e$.

In the last few years, the existence of energy ground states on metric graphs has been widely investigated. Since the literature on the subject is by now quite large and constantly growing (see e.g. \cite{ADST,AST,AST2,ASTcmp,BMP,BDL1,BDL2,BC,DT,KMPX,LLZ,PS,PSV} and references therein, and \cite{ABR,KNP} for comprehensive reviews), we do not even try to revise it. To provide a general overview, we simply recall that energy ground states never exist in the $L^2$-supercritical regime $p>6$ (where $\EE_{p,\G}(\mu)=-\infty$ for every $\mu$ and $\G$), whereas, in the $L^2$-subcritical $p\in(2,6)$ and $L^2$-critical $p=6$ cases, existence is strongly sensitive to the topology and the metric of the graph, and to the actual value of the mass. In particular, for our purposes here it is relevant to point out for what graphs ground states exist for every mass $\mu>0$ and every $L^2$-subcritical power $p\in(2,6)$, independently of the length of the edges. The largest family where this occurs is that of compact graphs (i.e. graphs with finitely many vertices and edges, all bounded, see \cite[Theorem 1.1]{D18}). In the noncompact setting, on the contrary, such a general existence result is very unlikely. For noncompact graphs with finitely many edges (at least one of which unbounded), the few structures for which this is known to happen are those in Figure \ref{fig:GShalf} below, identified in \cite[Section 3]{AST2}. For noncompact graphs with infinitely many edges, many papers proved that normalized ground states with small mass usually do not exist when $p$ is slightly less than $6$ (see e.g. \cite{ADST,DT,DSTjlms}). The only known exception is given by $\Z$-periodic graphs, i.e. graphs obtained gluing together in a $\Z$-symmetric pattern infinitely many copies of a given compact graph (for a rigorous definition see \cite[Definition 4.1.1]{BK} or \cite[Section 2]{D19}). In this latter case, existence is guaranteed for every $\mu>0$ and $p\in(2,6)$ by \cite[Theorem 1.1]{D19}.

\begin{figure}[t]
	\centering
	\subfloat[][]{
		\begin{tikzpicture}[xscale= 0.5,yscale=0.5]
			\node at (0,0) [nodo] (00) {};
			\node at (-7,0) [infinito] (-55) {};
			\draw[-] (00)--(-55);
			\node at (-7.5,0) [infinito] {$\footnotesize\infty$};
			\draw (2,0) circle (2); 
		\end{tikzpicture}
	}\qquad
	\subfloat[][]{
		\begin{tikzpicture}[xscale= 0.5,yscale=0.5]
			\node at (0,0) [nodo] (00) {};
			\node at (0,2) [nodo] (02) {};
			\node at (-4,-5) [infinito] (inf1) {};
			\node at (-1.5,-6) [infinito] (inf2) {};
			\node at (1.5,-6) [infinito] (inf3) {};
			\node at (4,-5) [infinito] (inf4) {};
			\draw[-] (00)--(02);
			\draw[-] (00)--(inf1);
			\draw[-] (00)--(inf2);
			\draw[-] (00)--(inf3);
			\draw[-] (00)--(inf4);
			\node at (-4,-5.2) [infinito] {$\footnotesize\infty$};
			\node at (-1.5,-6.2) [infinito] {$\footnotesize\infty$};
			\node at (1.5,-6.2) [infinito] {$\footnotesize\infty$};
			\node at (4,-5.2) [infinito] {$\footnotesize\infty$};
	\end{tikzpicture}}
	\caption{A tadpole graph (\textsc{\tiny A}) and a graph with some half-lines and one bounded edge glued together at the same vertex (\textsc{\tiny B}).}
	\label{fig-NlT}
\end{figure}
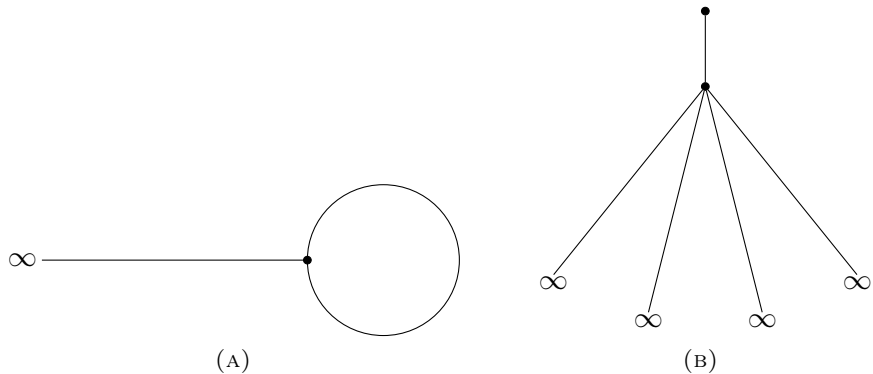

Contrary to existence, very little is known on uniqueness of normalized ground states. This is no surprise, since uniqueness issues for NLS equations are known to be very challenging also in contexts different than metric graphs. The problem may become even harder on graphs, since the lack of symmetry and scale invariance of the domain does not allow to exploit the few techniques developed in the literature in similar settings. Moreover, when focusing on energy ground states, the first question towards uniqueness (or non-uniqueness) is whether all ground states with the same mass are solutions of the {\em same} NLS equation, since a priori the parameter $\lambda$ in \eqref{eq:nlse} may depend on the specific ground state at hand (and not only on its mass).

To the best of our knowledge, the first systematic study on uniqueness of normalized ground states on metric graphs is \cite{DST20}. To illustrate the main contributions of that paper, let
\[
\mathfrak{E}_{p,\G}(\mu):=\left\{u\in H_\mu^1(\G)\,:\, E_p(u,\G)=\EE_{p,\G}(\mu)\right\}
\]
be the set of all ground states of $E_p(\cdot,\G)$ with mass $\mu$, and set
\begin{equation}
\label{eq:L+-}
\Lambda_{p,\G}^-(\mu):=\inf_{u\in \mathfrak{E}_{p,\G}(\mu)}\LL_p(u,\G),\qquad \Lambda_{p,\G}^+(\mu):=\sup_{u\in \mathfrak{E}_{p,\G}(\mu)}\LL_p(u,\G)\,.
\end{equation}
In \cite[Theorem 2.5]{DST20}, it has been shown that, for every metric graph $\G$ and every $p\in(2,6]$, there exists an at most countable set $\displaystyle Z_{p,\G}\subset \left\{\mu>0\,:\,\mathfrak{E}_{p,\G}(\mu)\neq\emptyset \right\}$ such that
\[
\Lambda_{p,\G}^-(\mu)\neq\Lambda_{p,\G}^+(\mu)\iff \mu\in Z_{p,\G}\,,
\] 
namely energy ground states at the same mass solve the same NLS equation \eqref{eq:nlse} (i.e. with the same value of $\lambda$), for all but at most countably many values of the mass (those in $Z_{p,\G}$). 
This has then been exploited (see \cite[Theorems 2.8--2.9]{DST20}) to prove uniqueness of energy ground states, for all masses not in $ Z_{p,\G}$, on two simple graph structures: the tadpole graph (Figure \ref{fig-NlT}(\textsc{\tiny A})), and graphs with a finite number of half-lines and one bounded edge glued together at a common vertex (Figure \ref{fig-NlT}(\textsc{\tiny B})). 

Looking at the proof of \cite[Theorem 2.5]{DST20}, the set $Z_{p,\G}$ pops up because the argument therein proves first that $\Lambda_{p,\G}^\pm$ are strictly increasing functions, and then uses the abstract fact that a monotone function is continuous out of at most countably many points.  It is then quite natural to wonder whether the presence of this set is just a mere consequence of this specific argument (that could be eliminated with a different proof of the continuity of $\Lambda_{p,\G}^\pm$), or if it is an actual feature of the problem. In other words, one may raise the following question:
\begin{equation}
	\label{Q1}
	\text{is }Z_{p,\G}=\emptyset\text{ for every }\G\text{ and every }p\in(2,6)?\tag{Q1}
\end{equation}

\medskip

\noindent Unfortunately, the answer to \eqref{Q1} is negative: \cite[Theorem 2.11]{DST20} shows that, for every fixed $p\in(2,6)$ and $\mu>0$, there exists at least one graph $\G$ such that $\mu\in Z_{p,\G}$. 

Even though this guarantees that there exists at least one graph where the non-uniqueness of $\lambda$ in \eqref{eq:nlse} corresponding to $Z_{p,\G}\neq\emptyset$ does occur, the explicit construction developed in \cite[Theorem 2.11]{DST20} is technically demanding and relies on a delicate combination of ad hoc topological and metric features. Actually, that \eqref{Q1} can be answered in the negative by a sufficiently involved counterexample is not so unexpected, given the high level of generality of the question and the wide variety of possible graph structures that one can conceive. However, this gives no insight on whether a non-empty $Z_{p,\G}$ is a pathological, unlikely event, that happens only on suitably complicated graphs. To this extent, it is perhaps more meaningful to replace \eqref{Q1} with the weaker version:
\begin{equation}
\label{Q}
\text{is there any class of graphs }\G\text{ for which }Z_{p,\G}=\emptyset\text{ for every }p\in(2,6)?\tag{Q2}
\end{equation} 

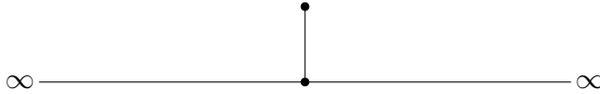
\begin{figure}[t!]
	\centering
	\begin{tikzpicture}[xscale=0.5,yscale=0.5]
	\node at (0,0) [nodo] {};
	\node at (0,2) [nodo] {};
	\draw (0,0)--(0,2);
	\draw (0,0)--(-7,0);
	\draw (0,0)--(7,0);
	\node at (-7.5,0) [infinito] {$\footnotesize\infty$};
	\node at (7.5,0) [infinito] {$\footnotesize\infty$};
	\end{tikzpicture}
	\caption{A $\mathcal{T}$-graph.}
	\label{fig:T}
\end{figure} 

\smallskip
\noindent Turning from \eqref{Q1} to \eqref{Q} is not pointless a priori, since it is well-known (see e.g. \cite{LC}) that on the half-line and on the line (graphs with one vertex and one/two unbounded edges, respectively) \eqref{Q} has a positive answer. Partial affirmative answers are available also on the tadpole graph, where normalized ground states are known to be unique, for every mass for which they exist, when $p=4$ (see \cite[Remark 1.3]{KMPX}) and $p=6$ (see \cite[Remark 1]{NP}). However, the recent paper \cite{ACT} unexpectedly showed that, if a class of graphs answering  \eqref{Q} in the affirmative exists, it does not contain one of the simplest graphs possible, the $\mathcal{T}$-graph (Figure \ref{fig:T}) given by one bounded edge and two half-lines glued together at the same vertex. Precisely, \cite[Theorem 1.5]{ACT} proved that, when $\G$ is the $\mathcal{T}$-graph, there exists $\varepsilon>0$ such that $Z_{p,\G}\neq\emptyset$ for every $p\in(6-\varepsilon,6)$. 
In view of the elementary structure of the $\mathcal{T}$-graph, at first sight this result is very surprising, and it is perhaps even more unforeseen if we think that $\mathcal{T}$-graphs with very short bounded edges could be somehow interpreted as small perturbations of the line.

\begin{figure}[t]
	\centering
	\subfloat[ ][ ]{
		\begin{tikzpicture}[xscale=0.6,yscale=0.6]
			
			\node at (0,0) [nodo] {};
			\node at (0,-1) [nodo] {};
			\draw (0,1) circle (1); 
			\draw (0,-1)--(0,0);
			\draw (0,-1)--(-5,-1); 
			\draw (0,-1)--(5,-1);
			\node at (-5.5,-1) [infinito] {$\footnotesize\infty$};
			\node at (5.5,-1) [infinito] {$\footnotesize\infty$};
			
		\end{tikzpicture}
	}
	
	\subfloat[][]{
		\begin{tikzpicture}[xscale=0.7,yscale=0.7]
			\node at (-5.5,-1) [infinito] {$\footnotesize\infty$};
			\node at (0,-1) [nodo] {};
			\draw (0,-1)--(-5,-1);
			\node at (2,0) [nodo] {};
			\node at (1,-2) [nodo] {};
			\draw (2,0)--(0,-1);
			\draw (1,-2)--(0,-1);
		\end{tikzpicture}
	}\qquad\qquad
	\subfloat[][]{
		\begin{tikzpicture}[xscale=0.7,yscale=0.7]
			\node at (-5.5,-1) [infinito] {$\footnotesize\infty$};
			\node at (0,-1) [nodo] {};
			\draw (0,-1)--(-5,-1);
			\node at (.5,0.1) [nodo] {};
			\node at (1.4,-1.5) [nodo] {};
			\draw (.5,0.1)--(0,-1);
			\draw (1.4,-1.5)--(0,-1);
			\node at (2,0) [nodo] {};
			\draw (0,-1)--(2,0);
		\end{tikzpicture}
	}
	\caption{Noncompact graphs with finitely many edges hosting $L^2$-subcritical normalized ground states for every mass (see \cite[Section 3]{AST2}): the signpost graph (\textsc{\tiny A}), the 2-fork graph (\textsc{\tiny B}) and the 3-fork graph (\textsc{\tiny C}).}
	\label{fig:GShalf}
\end{figure}
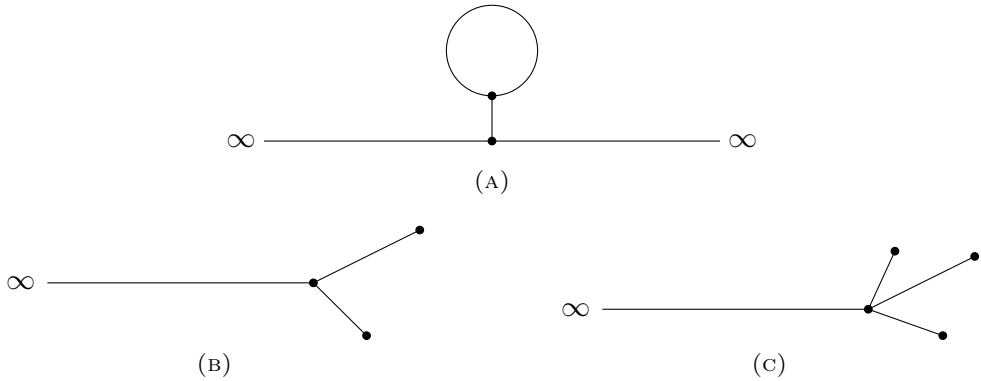

\smallskip
In this paper we give a general {\em negative} answer to \eqref{Q}. Roughly, we show that, except for the half-line and the line (and its equivalent version in the towers of bubbles \cite[Example 2.4]{AST}), all graphs where energy ground states exist for every $\mu>0$ and every $p\in(2,6)$ behave as $\mathcal{T}$-graphs, i.e. $Z_{p,\G}\neq\emptyset$ when $p$ is close to $6$.  

As we said before, the most general class of graphs where $L^2$-subcritical normalized ground states always exist, regardless of any other topological or metrical property, is that of compact graphs. Therefore, this is the best setting in which to state the first main result of the paper.

\begin{theorem}
	\label{thm:compact}
	For every compact graph $\G$ there exists a value $p_\G>2$ (depending on $\G$) such that $Z_{p,\G}\neq\emptyset$ for every $p\in[p_\G,6)$.
\end{theorem}

Hence, the answer to \eqref{Q} is {\em always} negative on compact graphs. In particular, for every compact graph and every nonlinearity power sufficiently close to 6, there exists at least one value of the mass $\mu$ such that energy ground states in $H_\mu^1(\G)$ are non-unique. Remarkably, this phenomenon does not depend on anything specific of the graph, though its topology and metric will clearly affect the actual value of $p_\G$ and of the mass where non-uniqueness takes place. Theorem \ref{thm:compact} thus unravels an intrinsic feature of normalized ground states on metric graphs: far from being a technical accident, a non-empty $Z_{p,\G}$ is essentially ubiquitous for slightly $L^2$-subcritical exponents. This situation is quite unsuspected, especially if one thinks of very simple graph structures. To give a concrete example, Theorem \ref{thm:compact} entails a genuine (not due to symmetry) non-uniqueness result for normalized ground states on a line segment with homogeneous Neumann conditions at the endpoints, and this marks a strong difference with the problem on the segment with other boundary conditions (compare e.g. with \cite[Theorem 1.5(i)]{NTV}).

We stress that the compactness of the space is by no means necessary, as the phenomenology of Theorem \ref{thm:compact} is expected to occur whenever energy ground states exist for every mass $\mu>0$ and every $L^2$-subcritical power $p\in(2,6)$. 
As anticipated above, beside the line, the half-line, the towers of bubbles, the tadpole graph and the $\mathcal{T}$-graph, the only noncompact graphs with finitely many edges known to always host normalized ground states in the $L^2$-subcritical regime are those of \cite[Section 3]{AST2}, here depicted in Figure \ref{fig:GShalf}. For all these graphs, the following analogous of Theorem \ref{thm:compact} holds. Observe in particular that the tadpole makes no exception, even though it is already known that uniqueness always takes place at $p=4$ and $p=6$.
\begin{theorem}
	\label{thm:noncomp}
	Let $\G$ be a noncompact graph as in Figure \ref{fig-NlT}(\textsc{\tiny A}) or Figure \ref{fig:GShalf}. Then there exists a value $p_\G>2$ (depending on $\G$) such that $Z_{p,\G}\neq\emptyset$ for every $p\in[p_\G,6)$.
\end{theorem}

In the case of infinitely many edges with uniformly bounded length, the only graphs where we always have $L^2$-subcritical ground states are the $\Z$-periodic ones. Since this family of graphs is rather wide, for the sake of simplicity we limit ourselves to consider a prototypical example, the {\em ladder} graph (Figure \ref{fig:Zper}).

\begin{theorem}
	\label{thm:per}
	Let $\G$ be a ladder graph (Figure \ref{fig:Zper}). Then there exists a value $p_\G>2$ (depending on $\G$) such that $Z_{p,\G}\neq\emptyset$ for every $p\in[p_\G,6)$.
\end{theorem}
Notice that Theorem \ref{thm:per} has nothing to do with the invariance of the graph under the action of the symmetry group $\Z$ (which is of course a source of non-uniqueness), but it unravels again a genuine non-uniqueness phenomenon for ground states solving \eqref{eq:nlse} with different values of $\lambda$.

\begin{figure}[t]
	\centering
	\includegraphics[width=0.4\columnwidth]{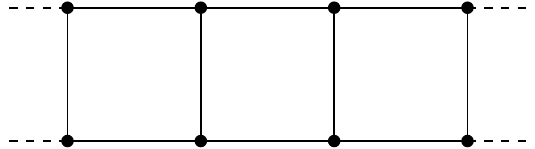}
	\caption{A prototypical example of $\Z$-periodic graph: the ladder graph.}
	\label{fig:Zper}
\end{figure}
\smallskip
Since Theorems \ref{thm:compact}--\ref{thm:noncomp}--\ref{thm:per} may appear surprising, let us conclude this introduction trying to provide a technical motivation for the occurrence of this phenomenology on graphs.

To this end, recall first that there is at least another notion of ground state associated to the NLS equation \eqref{eq:nlse}. Indeed, for any fixed $\lambda\in\R$, an {\em action ground state} of \eqref{eq:nlse} is a function $u$ such that
\begin{equation}
\label{eq:infA}
J_{p,\lambda}(u,\G)=\JJ_{p,\G}(\lambda)=:\inf_{v\in\NN_{p,\lambda}(\G)}J_{p,\lambda}(v,\G)\,,
\end{equation}
where the action functional $J_{p,\lambda}:H^1(\G)\to\R$ is given by
\begin{equation}
\label{action}
J_{p,\lambda}(v,\G):=\frac12\|v'\|_{L^2(\G)}^2+\frac{\lambda}{2}\|v\|_{L^2(\G)}^2-\frac1p\|v\|_{L^p(\G)}^p\,,
\end{equation}
$\NN_{p,\lambda}(\G)$ is the associated Nehari manifold 
\begin{equation}
	\label{nehari}
	\begin{split}
	\NN_{p,\lambda}(\G):=&\,\left\{v\in H^1(\G)\setminus\left\{0\right\}\,:\,J_{p,\lambda}'(v,\G)v=0\right\}\\
	=&\,\left\{v\in H^1(\G)\setminus\left\{0\right\}\,:\,\|v'\|_{L^2(\G)}^2+\lambda\|v\|_{L^2(\G)}^2=\|v\|_{L^p(\G)}^p\right\}
	\end{split}
\end{equation}
and $\JJ_{p,\G}$ is the action ground state level. It is standard knowledge that the Nehari manifold is a natural constraint for the action functional, i.e. critical points of $J_{p,\lambda}$ constrained to $\NN_{p,\lambda}(\G)$ are actually unconstrained critical points of $J_{p,\lambda}$, and thus solve \eqref{eq:nlse}. However, even though both ground states provide (up to sign) positive solutions to the same NLS equations, general investigations of the relation between the energy and the action approach started only recently in \cite{DST23, JL}. Clearly, contrary to energy ground states, nothing is known a priori on the mass of action ground states, and in principle different action ground states with the same $\lambda$ may have different masses. Hence, denoting by
\begin{equation}
	\label{eq:Y}
\mathfrak{A}_{p,\G}(\lambda):=\left\{u\in\NN_{p,\lambda}(\G)\,:\, J_{p,\lambda}(u,\G)=\JJ_{p,\G}(\lambda)\right\}
\end{equation}
the set of action ground states in $\NN_{p,\lambda}(\G)$, we set
\begin{equation}
\label{eq:Mpm}
M_{p,\G}^-(\lambda):=\inf_{u\in \mathfrak{A}_{p,\G}(\lambda)}\|u\|_{L^2(\G)}^2,\qquad M_{p,\G}^+(\lambda):=\sup_{u\in \mathfrak{A}_{p,\G}(\lambda)}\|u\|_{L^2(\G)}^2\,.
\end{equation}
The key-point in the proof of Theorems \ref{thm:compact}--\ref{thm:noncomp}--\ref{thm:per} will be the validity of the following property:
\begin{equation}
\label{P}
\text{there exist }\lambda_1<\lambda_2\text{ such that } M_{6,\G}^-(\lambda_1)>M_{6,\G}^+(\lambda_2)\,.\tag{P}
\end{equation}
As the details in the next sections will show, there are slightly different technical reasons for which \eqref{P} holds on the various graphs covered by our theorems, but they all are peculiar of the behaviour of the ground state problems at the $L^2$-critical power $p=6$ on metric graphs. To some extent, this explains somehow why the non-uniqueness of normalized ground states described here is ubiquitous on graphs, whilst we do not expect to observe it in other contexts (see e.g. Remark \ref{rem:Dircomp} below).

We also note that, as a by-product of our analysis, we obtain the following.
\begin{theorem}
	\label{thm:AnoE}
	For every graph $\G$ covered by Theorems \ref{thm:compact}--\ref{thm:noncomp}--\ref{thm:per} and every $p\in[p_\G,6)$, there exist action ground states that are not energy ground states with their mass.
\end{theorem}   
The proof of Theorem \ref{thm:AnoE} is a straightforward consequence of Theorems \ref{thm:compact}--\ref{thm:noncomp}--\ref{thm:per} and \cite[Theorem 1.3]{DST23}. This result has to be compared with \cite[Theorem 1.3]{DST23}, where it is proved that, on the contrary, energy ground states are always also action ground states. Hence, Theorem \ref{thm:AnoE} shows that, essentially on every metric graph, the action and the energy ground state problems are not equivalent.

Finally, let us say clearly that the idea of exploiting \eqref{P} to obtain normalized ground states with the same mass but different $\lambda$ for $p$ close to 6 was used for the first time on $\mathcal{T}$-graphs in the proof of \cite[Theorem 1.5]{ACT}, and indeed the phenomenology identified by that theorem and Theorems  \ref{thm:compact}--\ref{thm:noncomp}--\ref{thm:per} here is the same. However, as we already said, \cite{ACT} makes unavoidable use of thorough explicit computations that cannot work outside $\mathcal{T}$-graphs. Hence, to handle the level of generality we seek, we need to develop different arguments, as much independent as possible of any specific properties of the graph at hand. In particular, our proofs are based on fine asymptotic analyses of the action ground state problem at the $L^2$-critical power (see Proposition \ref{prop:est6} and Lemmas \ref{lem:pend}--\ref{lem:GtoR} below), together with a certain amount of continuity of the quantities under exam (whose derivation e.g. on the ladder graph as in Lemma \ref{lem:contp} below will require a rather detailed discussion).

\smallskip
The remainder of the paper is organized as follows. To ease the presentation and improve the readability, we first develop in full details the proof of our non-uniqueness results on compact graphs. For this reason, all the statements in Sections \ref{sec:prel}--\ref{sec:est6}--\ref{sec:comp} are given for compact graphs only: Section \ref{sec:prel} recalls some preliminaries about energy and action ground states, Section \ref{sec:est6} derives general asymptotic estimates for $L^2$-critical ground states, and Section \ref{sec:comp} completes the proof of Theorem \ref{thm:compact}. Sections \ref{sec:Ghalf}--\ref{sec:Zper} are then devoted to noncompact graphs with finitely many edges (Theorem \ref{thm:noncomp}) and $\Z$-periodic graphs (Theorem \ref{thm:per}) respectively. 

\medskip
{\bf Notation.} In what follows, whenever possible we use shorther notations like $\|u\|_p$ to denote $L^p$ norms, resorting to the complete ones only if necessary to avoid ambiguity.

\section{Preliminaries on action and energy ground states}
\label{sec:prel}

We begin here by recalling some preliminary results on action and energy ground states. As anticipated before, even though most of the properties listed below holds true in more general settings than that of compact graphs, for the moment we state everything only in this context to keep the focus on the proof of Theorem \ref{thm:compact}.

Let us first summarize what is known on the energy ground state problem.
\begin{proposition}
	\label{prop:Ecomp}
	Let $\G$ be a compact graph, $p\in(2,6]$ and $\EE_{p,\G}:[0,\infty)\to\R$ be the energy ground state level defined in \eqref{eq:infE}. There results:
	\begin{itemize}
		\item[(i)] if $p\in(2,6)$, then $\EE_{p,\G}$ is strictly negative and attained for every $\mu>0$;
		\item[(ii)] if $p=6$, then
		\[
		\EE_{6,\G}(\mu)\begin{cases}
		<0 & \text{if }\mu\leq\mu_\G\\
		=-\infty & \text{if }\mu>\mu_\G\,,
		\end{cases}
		\]
		and $\EE_{6,\G}$ is attained if and only if $\mu\leq\mu_\G$, with 
		\[
		\mu_\G=\begin{cases}
		\frac{\sqrt{3}}4\pi=:\mu_{\R^+} & \text{if $\G$ has a pendant}\\
		\frac{\sqrt{3}}2\pi=:\mu_\R & \text{if $\G$ has no pendant}\,.
		\end{cases}
		\]
	\end{itemize}
	Let then 
	\[
	I_{p,\G}:=\begin{cases}
	[0,\infty) & \text{for }p\in(2,6)\\
	[0,\mu_\G] & \text{for }p=6 
	\end{cases}
	\]
	and $\Lambda_{p,\G}^\pm$ be as in \eqref{eq:L+-}. Then $\Lambda_{p,\G}^\pm$ are attained for every $p\in(2,6]$ and every $\mu\in I_{p,\G}$, and there exists an at most countable set $Z_{p,\G}\subset I_{p,\G}$ such that
	\[
	\Lambda_{p,\G}^-(\mu)=\Lambda_{p,\G}^+(\mu)\qquad\forall\mu\in I_{p,\G}\setminus Z_{p,\G}\,.
	\]
	Moreover, the map $\Lambda_{p,\G}:I_{p,\G}\setminus Z_{p,\G}\to \R^+$, defined as $\Lambda_{p,\G}(\mu):=\Lambda_{p,\G}^\pm(\mu)$, is strictly increasing and
	\begin{equation}
		\label{eq:lim}
	\begin{split}
	\displaystyle\lim_{\mu\to0}\Lambda_{p,\G}(\mu)=0 & \qquad\forall p\in(2,6]\\
	\displaystyle\lim_{\mu\to\infty}\Lambda_{p,\G}(\mu)=\infty & \qquad\forall p\in(2,6)\,.
	\end{split}
	\end{equation}
\end{proposition}
\begin{proof}
	Points (i)--(ii) are the content of \cite[Theorems 1.1--1.2]{D18}, whereas that $\Lambda_{p,\G}^\pm$ are attained, the properties of the set $Z_{p,\G}$ and the monotonicity of the map $\Lambda_{p,\G}$ can be proved following verbatim the argument in the proof of \cite[Theorem 2.5]{DST20} (where the same results are obtained for noncompact graphs with finitely many edges). Finally, to see that $\Lambda_{p,\G}(\mu)\to0$ as $\mu\to0$ and $\Lambda_{p,\G}(\mu)\to\infty$ as $\mu\to\infty$, recall the identity
	\begin{equation}
	\label{eq:LE}
	\Lambda_{p,\G}(\mu)=\left(1-\frac2p\right)\frac{\|u\|_p^p}\mu-\frac{2\EE_{p,\G}(\mu)}\mu\,,
	\end{equation}
	where $u\in H_\mu^1(\G)$ is any ground state of $E_p(\cdot, \G)$ with mass $\mu$. Since the constant function $\displaystyle\kappa_\mu:=\sqrt{\mu/|\G|}$ (where $|\G|$ denotes the total length of $\G$) always belongs to $H_\mu^1(\G)$, we have $\displaystyle\EE_{p,\G}(\mu)\leq E_p(\kappa_\mu,\G)=-\mu^{\frac p2}/\left(p|\G|^{\frac p2-1}\right)$. As $p>2$, directly plugging into \eqref{eq:LE} shows that $\Lambda_{p,\G}(\mu)\to\infty$ when $\mu\to\infty$, and the same can be done in the case $\mu\to0$, recalling that by \cite[Theorem 2.2]{CDS} $\kappa_\mu$ is a ground state of $E_p(\cdot,\G)$ as soon as $\mu$ is small enough.
\end{proof}

An analogous characterization can be given for action ground states. Before it, let us introduce the notation (frequently used in the following)
\begin{equation}
\label{eq:sigma}
\sigma_{p,\lambda}(u):=\left(\frac{\|u'\|_2^2+\lambda\|u\|_2^2}{\|u\|_p^p}\right)^{\frac1{p-2}}
\end{equation}
for any given $p\in(2,\infty)$, $\lambda\in\R$ and $u\in H^1(\G)$ for which $\|u'\|_2^2+\lambda\|u\|_2^2>0$. Observe that, by definition, $\sigma_{p,\lambda}(u)u\in\NN_{p,\lambda}(\G)$.
\begin{proposition}
	\label{prop:Acomp}
	Let $\G$ be a compact graph, $p>2$ and $\JJ_{p,\G}:\R\to\R$ be the action ground state level defined in \eqref{eq:infA}. Then
	\begin{equation}
	\label{eq:levJ}
	\JJ_{p,\G}(\lambda)\begin{cases}
	>0 & \text{if }\lambda>0\\
	=0 & \text{if }\lambda\leq0\,,
	\end{cases}
	\end{equation}
	$\JJ_{p,\G}$ is continuous on $\R$, strictly increasing on $\R^+$, and it is attained if and only if $\lambda>0$. Let then $M_{p,\G}^\pm$ be as in \eqref{eq:Mpm}. Then $M_{p,\G}^\pm$ are attained for every $\lambda>0$ and there exists an at most countable set $\widetilde{Z}_{p,\G}\subset\R^+$ such that
	\[
	M_{p,\G}^-(\lambda)=M_{p,\G}^+(\lambda)\qquad\forall \lambda\in \R^+\setminus\widetilde{Z}_{p,\G}\,.
	\]
	Furthermore, the righ and left derivatives of $\JJ_{p,\G}$ exist everywhere on $\R^+$ and satisfy
	\[
	\left(\JJ_{p,\G}\right)_-'(\lambda)=\frac{M_{p,\G}^+(\lambda)}2,\qquad \left(\JJ_{p,\G}\right)_+'(\lambda)=\frac{M_{p,\G}^-(\lambda)}2\qquad\forall\lambda>0\,.
	\]
	In particular, $\JJ_{p,\G}$ is locally Lipschitz on $\R^+$ and differentiable in $\R^+\setminus\widetilde{Z}_{p,\G}$, with $\JJ_{p,\G}'(\lambda)=M_{p,\G}^\pm(\lambda)/2$ for every $\lambda\in\R^+\setminus\widetilde{Z}_{p,\G}$.
\end{proposition}
\begin{proof}
	The proof is divided in three steps.
	
	\smallskip
	{\em Step 1: $\JJ_{p,\G}(\lambda)$ when $\lambda>0$}. Observe first that
	\begin{equation}
	\label{eq:JonN} 
	J_{p,\lambda}(u,\G)=\left(\frac12-\frac1p\right)\|u\|_p^p=\left(\frac12-\frac1p\right)\left(\|u'\|_2^2+\lambda\|u\|_2^2\right)
	\end{equation}
	for every $u\in\NN_{p,\lambda}(\G)$. If $\lambda>0$, Sobolev embeddings of $H^1(\G)$ into $L^p(\G)$ give
	\[
	\|u\|_{H^1}^p\geq C\|u\|_p^p=C\left(\|u'\|_2^2+\lambda\|u\|_2^2\right)\geq C\min\left\{\lambda,1\right\}\|u\|_{H^1}^2, 
	\]
	so that $p>2$ and \eqref{eq:JonN} imply $\JJ_{p,\G}(\lambda)>0$ for every $\lambda>0$. The fact that $\JJ_{p,\G}$ is attained then easily follows by the compactness of $\G$. Indeed, if $(u_n)_n\subset\NN_{p,\lambda}(\G)$ satisfies $J_{p,\lambda}(u_n,\G)\to\JJ_{p,\G}(\lambda)$ as $n\to\infty$, then $(u_n)_n$ is bounded in $H^1(\G)$ by \eqref{eq:JonN}. Hence, there exists $u\in H^1(\G)$ such that (up to subsequences) $u_n\rightharpoonup u$ in $H^1(\G)$ and $u_n\to u$ in $L^q(\G)$ for every $q\geq2$, since the embedding of $H^1(\G)$ into $L^q(\G)$ is compact. This implies that $u\not\equiv0$ on $\G$ (as $\displaystyle \|u\|_p^p=\lim_{n\to\infty}\|u_n\|_p^p=2p\JJ_{p,\G}(\lambda)/(p-2)>0$) and, by weak lower semicontinuity,
	\[
	\sigma_{p,\lambda}(u)\leq \lim_{n\to\infty}\left(\frac{\|u_n'\|_2^2+\lambda\|u_n\|_2^2}{\|u_n\|_p^p}\right)^\frac1{p-2}=1\,,
	\]
	in turn yielding (since $\sigma_{p,\lambda}(u)u\in\NN_{p,\lambda}(\G)$)
	\[
	\JJ_{p,\G}(\lambda)\leq J_{p,\lambda}(\sigma_{p,\lambda}(u)u,\G)=\left(\frac12-\frac1p\right)\sigma_{p,\lambda}(u)^p\|u\|_p^p\leq\left(\frac12-\frac1p\right)\lim_{n\to\infty}\|u_n\|_p^p=\JJ_{p,\G}(\lambda)\,,
	\]
	i.e. $u\in\NN_{p,\lambda}(\G)$ and $\JJ_{p,\G}(\lambda)=J_{p,\lambda}(u,\G)$.
	
	\smallskip
	{\em Step 2: $\JJ_{p,\G}(\lambda)$ when $\lambda\leq0$}. For any fixed $\lambda\leq0$, there exist $k\in\N$ and $t\in[0,1]$ such that $-\lambda=t\lambda_k+(1-t)\lambda_{k+1}$, where $\lambda_k,\lambda_{k+1}$ are the $k$-th and $(k+1)$-th eigenvalues of the operator $\displaystyle -d^2/dx^2$ on $\G$ endowed with homogeneous Kirchhoff vertex conditions. Let $\varphi_k,\varphi_{k+1},\varphi_{k+2}\in H^1(\G)$ be eigenfunctions associated to $\lambda_k,\lambda_{k+1},\lambda_{k+2}$ respectively, such that $\|\varphi_i\|_2=1$, for every $i=k,k+1,k+2$.  For every $\varepsilon>0$, set then $v_\varepsilon:=\sqrt{t}\varphi_k+\sqrt{1-t}\varphi_{k+1}+\varepsilon\varphi_{k+2}$, so that
	\[
	\begin{split}
		\|v_\varepsilon'\|_2^2+\lambda\|v_\varepsilon\|_2^2=&\,t\|\varphi_k'\|_2^2+(1-t)\|\varphi_{k+1}'\|_2^2+\varepsilon^2\|\varphi_{k+2}'\|_2^2\\
		&\,+t\lambda\|\varphi_k\|_2^2+\lambda(1-t)\|\varphi_{k+1}\|_2^2+\varepsilon^2\lambda\|\varphi_{k+2}\|_2^2\\
		=&\,t\lambda_k\|\varphi_k\|_2^2+(1-t)\lambda_{k+1}\|\varphi_{k+1}\|_2^2+\varepsilon^2\lambda_{k+2}\|\varphi_{k+2}\|_2^2\\
		&\,+t\lambda\|\varphi_k\|_2^2+\lambda(1-t)\|\varphi_{k+1}\|_2^2+\varepsilon^2\lambda\|\varphi_{k+2}\|_2^2\\
		=&\,t\lambda_k+(1-t)\lambda_{k+1}+\lambda+\varepsilon^2(\lambda_{k+2}+\lambda)=\varepsilon^2(\lambda_{k+2}+\lambda)>0\,,
	\end{split}
	\]
	where we used the orthogonality of eigenfunctions, the definition of $t$ and the fact that $-\lambda\leq\lambda_{k+1}<\lambda_{k+2}$ (up to multiplicity of the eigenvalues, we can always assume the second inequality to be strict with no loss of generality). Hence, taking
	\[
	w_\varepsilon:=\sigma_{p,\lambda}(v_\varepsilon)v_\varepsilon=\varepsilon^{\frac2{p-2}}\left(\frac{\lambda+\lambda_{k+2}}{\|v_\varepsilon\|_p^p}\right)^{\frac1{p-2}}v_\varepsilon\,,
	\]
	we obtain $w_\varepsilon\in\NN_{p,\lambda}(\G)$ and (recalling \eqref{eq:JonN} and $\|v_\varepsilon\|_p\to\|\sqrt{t}\varphi_k+\sqrt{1-t}\varphi_{k+1}\|_p>0$ as $\varepsilon\to0$)
	\[
	0\leq\JJ_{p,\G}(\lambda)\leq\lim_{\varepsilon\to0}J_{p,\lambda}(w_\varepsilon,\G)=0
	\]
	that completes the proof of \eqref{eq:levJ}.
	
	\smallskip
	{\em Step 3: monotonicity and continuity of $\JJ_{p,\G}$ and properties of $M_{p,\G}^\pm$}. The monotonicity and continuity of $\JJ_{p,\G}$ are standard on compact graphs and can be proved arguing exactly as in \cite[Lemma 2.4(ii) and Remark 2.5]{DST23}. The equality of $M_{p,\G}^-$ and $M_{p,\G}^+$ for every $\lambda>0$ out of an at most countable set $\widetilde{Z}_{p,\G}$ and the properties of the derivatives of $\JJ_{p,\G}$ follow repeating with no modification the argument in the proof of \cite[Theorem 1.5]{DST23} (where the same result is proved for open subsets of $\R^N$ with homogeneous Dirichlet boundary conditions), and then the local Lipschitz continuity is guaranteed by the local boundedness of $\JJ_{p,\G}$, the explicit formulas for its derivatives and \eqref{eq:JonN}. Finally, that $M_{p,\G}^\pm(\lambda)$ are attained for every $\lambda>0$ is ensured again by the compactness of $\G$, since any sequence $(u_n)_n\subset \mathfrak{A}_{p,\G}(\lambda)$ (the set $\mathfrak{A}_{p,\G}(\lambda)$ being defined in \eqref{eq:Y}) is in particular a minimizing sequence for $J_{p,\lambda}$ in $\NN_{p,\lambda}(\G)$, and it thus converges strongly in $H^1(\G)$ (up to subsequences) to some $u\in\mathfrak{A}_{p,\G}(\lambda)$.
\end{proof}

\begin{remark}
	\label{rem:comp}
	Note that,  by definition of $M_{p,\G}^\pm$, for every given $\overline{p}\in(2,\infty)$ and $\lambda>0$ we have
	\begin{equation}
		\label{eq:boundM}
		\liminf_{p\to\overline{p}}M_{p,\G}^-(\lambda)\geq M_{\overline{p},\G}^-(\lambda),\qquad\limsup_{p\to\overline{p}}M_{p,\G}^+(\lambda)\leq M_{\overline{p},\G}^+(\lambda)\,.
	\end{equation}
	To prove \eqref{eq:boundM}, it is actually enough to show that, for any fixed $\lambda$, the action ground state level $\JJ_{p,\G}(\lambda)$ is continuous as a function of $p\to\overline{p}\in(2,\infty)$. To see this, let $p_n\to\overline{p}$ as $n\to\infty$ and observe first that, as for any fixed $v\in\NN_{\overline{p},\lambda}(\G)$ such that $J_{\overline{p},\lambda}(v,\G)=\JJ_{\overline{p},\G}(\lambda)$ we have (by \eqref{eq:levJ} and the Dominated Convergence Theorem)
	\[
	\lim_{n\to\infty}\sigma_{p_n,\lambda}(v)^{p_n-2}=\lim_{n\to\infty}\frac{\|v\|_{\overline{p}}^{\overline{p}}}{\|v\|_{p_n}^{p_n}}=1\,,
	\]
	we obtain
	\begin{equation}
	\label{eq:complimsup}
	\limsup_{n\to\infty} \JJ_{p_n,\G}(\lambda)\leq\limsup_{n\to\infty}\left(\frac12-\frac1{p_n}\right)\sigma_{p_n,\lambda}(v)^{p_n}\|v\|_{p_n}^{p_n}=\left(\frac12-\frac1{\overline{p}}\right)\|v\|_{\overline{p}}^{\overline{p}}=\JJ_{\overline{p},\G}(\lambda)\,.
	\end{equation}
	Let then $u_n\in\NN_{p_n,\lambda}(\G)$ be such that $J_{p_n,\lambda}(u_n,\G)=\JJ_{p_n,\G}(\lambda)$. By \eqref{eq:complimsup}, $(u_n)_n$ is bounded in $H^1(\G)$, so that (up to subsequences) converges weakly in $H^1(\G)$ and strongly in $L^q(\G)$, for every $q\geq2$, to some $\overline{u}$. In particular, $\overline{u}\not\equiv0$ on $\G$, because the convergence of $u_n$ to $u$ is strong in $L^\infty(\G)$ and $u_n\geq\lambda^{\frac1{p_n-2}}$ at any of its local maximum points by \eqref{eq:nlse}.  Since boundedness in $H^1(\G)$ implies boundedness in $L^\infty(\G)$, the Dominated Convergence Theorem gives
	\begin{equation}
	\label{eq:ppbar}
	\int_\G\left||u_n(x)|^{p_n}-|\overline{u}|^{\overline{p}}(x)\,dx\right|\to0\qquad\text{as }n\to\infty\,,
	\end{equation}
	that, together with $\|\overline{u}\|_{\overline{p}}\neq0$ and the strong convergence of $u_n$ to $\overline{u}$ in $L^{\overline{p}}(\G)$, implies
	\[
	\sigma_{\overline{p},\lambda}(u_n)^{\overline{p}-2}=\frac{\|u_n\|_{p_n}^{p_n}}{\|u_n\|_{\overline{p}}^{\overline{p}}}\to 1\qquad\text{as }n\to\infty\,,
	\]
	in turn yielding
	\begin{equation}
	\label{eq:convJpnp}
	\JJ_{\overline{p},\G}(\lambda)\leq\left(\frac12-\frac1{\overline{p}}\right)\lim_{n\to\infty} \sigma_{\overline{p},\lambda}(u_n)^{\overline{p}}\|u_n\|_{\overline{p}}^{\overline{p}}=\lim_{n\to\infty} \left(\frac12-\frac1{p_n}\right)\|u_n\|_{p_n}^{p_n}=\lim_{n\to\infty} \JJ_{p_n,\G}(\lambda)\,.
	\end{equation}
	Combining \eqref{eq:complimsup} and \eqref{eq:convJpnp} proves the desired continuity in $p$ of the action ground state level. Moreover, the previous argument also ensures that the limit $\overline{u}\in\NN_{\overline{p},\lambda}(\G)$ satisfies $\JJ_{\overline{p},\G}(\lambda)=J_{\overline{p},\lambda}(\overline{u},\G)$ and that the convergence of $u_n$ to $\overline{u}$ is strong in $H^1(\G)$ (by \eqref{eq:JonN}), thus implying \eqref{eq:boundM}.	
\end{remark}

We conclude this section with a general relation between energy and action ground states that will play a crucial role in the proof of Theorem \ref{thm:compact}. The next proposition simply rephrases in the context of graphs the content of \cite[Theorem 1.3]{DST23}. Even though in \cite{DST23} the proof is developed for open subsets of $\R^N$, the argument is completely abstract and extends to graphs without changes.
\begin{proposition}
	\label{prop:AEgs}
	Let $\G$ be a compact graph, $p\in(2,6]$ and $\mu>0$. Assume that $u\in H_\mu^1(\G)$ is a ground state of $E_p(\cdot,\G)$ with mass $\mu$ and let $\lambda=\LL_p(u,\G)$ (defined as in \eqref{eq:Lu}) Then $u\in\NN_{p,\lambda}(\G)$ is also a ground state of $J_{p,\lambda}(\cdot,\G)$. Moreover, any other ground state $v\in\NN_{p,\lambda}(\G)$ of $J_{p,\lambda}(\cdot,\G)$ satisfies $\|v\|_2^2=\mu$ and it is also a ground state of $E_p(\cdot,\G)$ with mass $\mu$.
\end{proposition}

\section{Asymptotic estimates in the $L^2$-critical regime $p=6$}
\label{sec:est6}

This section is devoted to the analysis of the $L^2$-critical action ground state problem $\JJ_{6,\G}$ on compact graphs in the limit for $\lambda\to\infty$.

\begin{remark}
	\label{rem:lev6R}
	Recall that, for every $\lambda>0$, on the real line the action ground state level at $p=6$ is explicitly given by
	\[
	\JJ_{6,\R}(\lambda)=\frac{\mu_\R}2\lambda\,,
	\]
	where $\mu_\R$ is the number defined in Proposition \ref{prop:Ecomp}, and the unique (up to symmetries) action ground state in $\NN_{6,\lambda}(\R)$ is the soliton $\phi_\lambda(x):=\lambda^{\frac14}\phi_1(\sqrt{\lambda}\,x)$, where 
	\begin{equation}
		\label{eq:phi1}
	\phi_1(x)=\sqrt{\sqrt{3}\,\text{\normalfont sech}(2|x|)}
	\end{equation}
	is the unique positive solution in $H^1(\R)$ of 
	\[
	\begin{cases}
		u''+|u|^4u=u & \text{on }\R\\
		\|u\|_{\infty}=u(0)\,. & 
	\end{cases}
	\]
	Solitons satisfy $\|\phi_\lambda\|_2^2=\mu_\R$ and $E_6(\phi_\lambda,\R)=0=\EE_{6,\R}(\mu_\R)$, namely they are the unique $L^2$-critical energy ground states on $\R$. Similarly, on the half-line $\R^+$ there results
	\[
	\JJ_{6,\R^+}(\lambda)=\frac{\mu_{\R^+}}2\lambda\qquad\forall\lambda>0,
	\]
	with $\displaystyle\mu_{\R^+}=\mu_\R/2$, and action and energy ground states are the half-solitons, i.e. the restriction of $\phi_\lambda$ to $\R^+$.
\end{remark}
The main result of the section is the following proposition.
\begin{proposition}
	\label{prop:est6}
	Let $\G$ be a compact graph. For every $\varepsilon>0$ there exists $\widetilde{\lambda}>0$ (depending on $\varepsilon$) such that
	\begin{itemize}
		\item[(i)] if $\G$ has no pendant, then 
		\begin{equation}
		\label{eq:critmasR}
		\mu_\R-\varepsilon\leq M_{6,\G}^-(\lambda)\leq M_{6,\G}^+(\lambda)\leq\mu_\R+\varepsilon\qquad\forall\lambda\geq\widetilde{\lambda}\,;
		\end{equation}
		
		\item[(ii)] if $\G$ has at least one pendant, then 
		\begin{equation}
			\label{eq:critmasR+}
		\mu_{\R^+}-\varepsilon\leq M_{6,\G}^-(\lambda)\leq M_{6,\G}^+(\lambda)\leq\mu_{\R^+}+\varepsilon\qquad\forall\lambda\geq\widetilde{\lambda}\,.
		\end{equation}
			\end{itemize}
		Moreover, there exists $\gamma>0$ such that if $\G$ has no pendant, then
		\begin{equation}
		\label{eq:JGJR}
		\left|\JJ_{6,\G}(\lambda)-\JJ_{6,\R}(\lambda)\right|=o\left(e^{-\gamma\sqrt{\lambda}}\right)\qquad\text{as }\lambda\to\infty\,,
		\end{equation}
		whereas if $\G$ has at least one pendant, then
		\begin{equation}
		\label{eq:JGJR+}
		\left|\JJ_{6,\G}(\lambda)-\JJ_{6,\R^+}(\lambda)\right|=o\left(e^{-\gamma\sqrt{\lambda}}\right)\qquad\text{as }\lambda\to\infty\,.
		\end{equation}
\end{proposition}
\begin{proof}
	Let us start by proving the result for compact graphs with no pendant. Observe first that, if $u\in\NN_{6,\lambda}(\G)$ is a ground state of $J_{6,\lambda}(\cdot,\G)$, then
	\begin{equation}
	\label{eq:ul}
	u_\lambda(x):=\lambda^{-\frac14}u\big(x/\sqrt{\lambda}\big),
	\end{equation}
	defined on $\G_\lambda:=\sqrt{\lambda}\,\G$, satisfies $u_\lambda\in\NN_{6,1}(\G_\lambda)$ and it is a ground state of $J_{6,1}(\cdot,\G_\lambda)$. Moreover, direct computations give
	\begin{equation}
	\label{eq:GtoGl}
	\|u_\lambda\|_{L^2(\G_\lambda)}=\|u\|_{L^2(\G)}\,,\qquad \JJ_{6,\G}(\lambda)=\lambda \JJ_{6,\G_\lambda}(1)\,.
	\end{equation}
Hence, in view of Remark \ref{rem:lev6R}, to prove \eqref{eq:critmasR} and \eqref{eq:JGJR} it is enough to show that
\begin{equation}
\label{eq:levGl}
\left|\JJ_{6,\G_\lambda}(1)-\JJ_{6,\R}(1)\right|=o\left(e^{-\gamma\sqrt{\lambda}}\right)\qquad\text{as }\lambda\to\infty
\end{equation}
for some $\gamma>0$, and that
\begin{equation}
	\label{eq:massGl}
	 \|u_\lambda\|_{L^2(\G_\lambda)}^2\to\|\phi_1\|_{L^2(\R)}^2=\mu_\R\qquad\text{as }\lambda\to\infty
\end{equation}
uniformly on the set of ground states $u_\lambda\in\NN_{6,1}(\G_\lambda)$ of $J_{6,1}(\cdot, \G_\lambda)$.

The rest of the proof is divided into five steps.
\smallskip

{\em Step 1: estimate from above of $\JJ_{6,\G_{\lambda}}(1)$.}
Setting $\displaystyle\ell:=\min_{e\in\E_\G}|e|/2$, for sufficiently large $\lambda$ it holds
\begin{equation}
	\label{eq:estR1}
\JJ_{6,\G_\lambda}(1)\leq\JJ_{6,\R}(1)+O\left(e^{-2\ell\sqrt{\lambda}}\right)\,.
\end{equation}
Indeed, consider the function $v:\left[-\sqrt{\lambda}\,\ell,\sqrt{\lambda}\,\ell\right]\to\R$ defined by
\[
v(x):=\begin{cases}
\phi_1(x) & \text{if }x\in\big[-\sqrt{\lambda}\,\ell+1,\sqrt{\lambda}\,\ell-1\big]\\
\phi_1\big(\sqrt{\lambda}\,\ell-1\big)\left(\sqrt{\lambda}\,\ell-x\right) & \text{if }x\in\big[\sqrt{\lambda}\,\ell-1,\sqrt{\lambda}\,\ell\big]\\
\phi_1\big(-\sqrt{\lambda}\,\ell+1\big)\left(x+\sqrt{\lambda}\,\ell\right) & \text{if }x\in\big[-\sqrt{\lambda}\,\ell,-\sqrt{\lambda}\,\ell+1\big]\,,
\end{cases}
\]
where $\phi_1$ is the soliton on the real line given by \eqref{eq:phi1}. Since
\[
\begin{split}
&\|v\|_{L^2\left(-\sqrt{\lambda}\,\ell,\sqrt{\lambda}\,\ell\right)}^2=\|\phi_1\|_{L^2\left(-\sqrt{\lambda}\,\ell+1,\sqrt{\lambda}\,\ell-1\right)}^2+\frac23\phi_1\big(\sqrt{\lambda}\,\ell-1\big)^2\,,\\
&\|v'\|_{L^2\big(-\sqrt{\lambda}\,\ell,\sqrt{\lambda}\,\ell\big)}^2=\|\phi_1'\|_{L^2\big(-\sqrt{\lambda}\,\ell+1,\sqrt{\lambda}\,\ell-1\big)}^2+2\phi_1\big(\sqrt{\lambda}\,\ell-1\big)^2\,,\\
&\|v\|_{L^6\big(-\sqrt{\lambda}\,\ell,\sqrt{\lambda}\,\ell\big)}^6=\|\phi_1\|_{L^6\big(-\sqrt{\lambda}\,\ell+1,\sqrt{\lambda}\,\ell-1\big)}^6+\frac27\phi_1\big(\sqrt{\lambda}\,\ell-1\big)^6\,,
\end{split}
\]
by \eqref{eq:phi1} when $\lambda\to\infty$ we obtain
\[
\begin{split}
&\|v\|_{L^2\left(-\sqrt{\lambda}\,\ell,\sqrt{\lambda}\,\ell\right)}^2=\|\phi_1\|_{L^2(\R)}^2+O\left(e^{-2\ell\sqrt{\lambda}}\right)\\
&\|v'\|_{L^2\left(-\sqrt{\lambda}\,\ell,\sqrt{\lambda}\,\ell\right)}^2=\|\phi_1'\|_{L^2(\R)}^2+O\left(e^{-2\ell\sqrt{\lambda}}\right)\\
&\|v\|_{L^6\left(-\sqrt{\lambda}\,\ell,\sqrt{\lambda}\,\ell\right)}^6=\|\phi_1\|_{L^6(\R)}^6+o\left(e^{-2\ell\sqrt{\lambda}}\right)\,,
\end{split}
\]
in turn yielding (recall \eqref{eq:sigma}) $\sigma_{6,1}(v)^6=1+O\left(e^{-2\ell\sqrt{\lambda}}\right)$. Hence, the function $w:=\sigma_{6,1}(v)v\in\NN_{6,1}\left(-\sqrt{\lambda}\,\ell,\sqrt{\lambda}\,\ell\right)$ and it has compact support, so that we can think of it as a function in $\NN_{6,1}(\G_\lambda)$ compactly supported on a single edge of the graph, thus implying, as $\lambda\to\infty$,
\[
\begin{split}
\JJ_{6,\G_{\lambda}}(1)\leq J_{6,1}(w,\G_\lambda)&\,=\frac13\|w\|_6^6=\frac13\sigma_{6,1}(v)^6\|v\|_6^6\\
&\,=\frac13\left(1+O\left(e^{-2\ell\sqrt{\lambda}}\right)\right)\left(\|\phi_1\|_6^6+o\left(e^{-2\ell\sqrt{\lambda}}\right)\right)=\JJ_{6,\R}(1)+O\left(e^{-2\ell\sqrt{\lambda}}\right),
\end{split}
\]
that is \eqref{eq:estR1}.
\smallskip

{\em Step 2: asymptotic estimates for action ground states.}
Let $u_\lambda$ be an action ground state in $\NN_{6,1}(\G_\lambda)$ and assume $u_\lambda>0$ on $\G_\lambda$ (that this gives no loss of generality is standard, see e.g. \cite[Remark 3.1]{DDGST}). By \eqref{eq:JonN} and \eqref{eq:estR1}, there exists a constant $C>0$ independent of $\lambda$ such that $\|u_\lambda\|_6^6=\|u_\lambda'\|_2^2+\|u_\lambda\|_2^2\leq C$ as soon as $\lambda$ is large enough. Since the length of each edge of $\G_\lambda$ diverges as $\lambda\to\infty$, this yields
\begin{equation}
\label{eq:alpha1}
\alpha_1:=\max_{e\in\E_{\G_\lambda}}\min_{x\in e} u_\lambda(x)\to0\qquad\text{as }\lambda\to\infty
\end{equation}
uniformly on the set of action ground states in $\NN_{6,1}(\G_\lambda)$ (because if this were not the case, there would exist a sequence of action ground states $u_\lambda\in\NN_{6,1}(\G_\lambda)$ along which $\|u_\lambda\|_6^6\geq\alpha_1^6|\G|\sqrt{\lambda}\to\infty$ as $\lambda\to\infty$). Conversely, since $u_\lambda$ satisfies
\begin{equation}
\label{eq:NLSul}
\begin{cases}
u_\lambda''+u_\lambda^{5}=u_\lambda & \text{on each edge }e\in\E_{\G_\lambda}\\
\sum_{e\succ \vv}\frac{du_\lambda}{dx_e}(\vv)=0 & \text{for every vertex }\vv\in\V_{\G_\lambda}\\
u_\lambda>0 & \text{on }\G_\lambda\,,
\end{cases}
\end{equation}
any local maximum point $x\in\G_\lambda$ of $u_\lambda$ is such that
\begin{equation}
	\label{eq:maxu}
	u_\lambda(x)\geq1\,.
\end{equation}
Furthermore, if in \eqref{eq:maxu} equality holds at one point, then by uniqueness of the solution to the Cauchy problem solved by $u_\lambda$ we would have $u_\lambda\equiv1$ on all edges containing that point. But this is impossible for large $\lambda$, as it would imply $\|u_\lambda\|_6\to\infty$ as $\lambda\to\infty$. Hence, whenever $\lambda$ is large enough all local maximum points of $u_\lambda$ are isolated.

Observe also that
\begin{equation}
	\label{eq:alpha2}
\alpha_2:=\max_{\vv\in\V_{\G_\lambda}}u_\lambda(\vv)\to0\qquad\text{as }\lambda\to\infty
\end{equation}
uniformly on the set of action ground states in $\NN_{6,1}(\G_\lambda)$. Indeed, assume by contradiction that this is not the case. Then there exist numbers $\lambda_n\to\infty$ and action ground states $u_{\lambda_n}\in\NN_{6,1}(\G_{\lambda_n})$ along which
\begin{equation}
\label{eq:absV}
\liminf_{n\to\infty}\max_{\vv\in\V_{\G_{\lambda_n}}}u_{\lambda_n}(\vv)>0\,.
\end{equation}
With no loss of generality, let $\overline{\vv}\in\V_{\G_{\lambda_n}}$ be a vertex of $\G_{\lambda_n}$ such that $\displaystyle u_{\lambda_n}(\overline{\vv})=\max_{\vv\in\V_{\G_{\lambda_n}}}u_{\lambda_n}(\vv)$ (where with a slight abuse of notation we identify corresponding vertices in $\G_{\lambda_n}$ for different values of $\lambda_n$) and let $\deg(\overline{\vv})$ be its degree (i.e. the total number of edges emanating from $\overline{\vv}$), that (up to subsequences) is independent of $\lambda_n$ by the fact that $\G_{\lambda_n}=\sqrt{\lambda_n}\G$ and that $\G$ has finitely many vertices. Since $\G$ has no pendant, $\deg(\overline{\vv})\geq3$. For every $i=1,\dots,\deg(\overline{\vv})$, let $a_i^n>0$ denote the distance, along the $i$-th edge $e_i\in\E_{\G_{\lambda_n}}$ incident at $\overline{\vv}$, between $\overline{\vv}$ and the local minimum point of $u_{\lambda_n}$ on $e_i$ closest to $\overline{\vv}$. Since $u_{\lambda_n}$ satisfies \eqref{eq:NLSul} on $\G_{\lambda_n}$, by  phase-plane analysis, \eqref{eq:alpha1} and \eqref{eq:absV}, it follows that $a_i^n\to\infty$ as $n\to\infty$. Moreover, the restriction of $u_{\lambda_n}$ to the interval $[0,a_i^n]$ on $e_i$ starting at $\overline{\vv}$ is either monotone decreasing from $0$ to $a_i^n$, or monotone increasing from $0$ to a local maximum point in the interior of $[0,a_i^n]$ and then monotone decreasing from this point to $a_i^n$, and it satisfies
\[
\begin{cases}
u_{\lambda_n}''+u_{\lambda_n}^5=u_{\lambda_n} & \text{on }[0,a_i^n]\\
u_{\lambda_n}>0, \quad u_{\lambda_n}(0)\geq K
\end{cases}
\]
with $u_{\lambda_n}(a_i^n)\to0$ as $n\to\infty$, for some $K>0$ independent of $n$ by \eqref{eq:absV}. By continuity in the phase-plane, as $n\to\infty$ we then have that the restriction of $u_{\lambda_n}$ to each interval $[0,a_i^n]$ converges in $C^2_{\text{\normalfont loc}}$ to the restriction of the soliton $\phi_1$ to a suitable interval of the form $[a_i,\infty)$. Therefore, the union over $i=1,\dots,\deg(\overline{\vv})$ of these restrictions of $u_{\lambda_n}$  converges to a strictly positive $H^1$ solution of the problem
\begin{equation}
	\label{eq:nlsSTAR}
\begin{cases}
	v''+v^5=v & \text{on each edge of }S_{\deg{\overline{\vv}}}\\
	\sum_{i=1}^{\deg(\overline{\vv})}\frac{dv}{dx_i}(0)=0 & 
\end{cases}
\end{equation}
on the infinite star graph $S_{\deg(\overline{\vv})}$ with $\deg(\overline{\vv})$ half-lines. But this is impossible, since it is well-known (see e.g. \cite{ACFN}) that any $H^1$ solution $v$ of \eqref{eq:nlsSTAR} satisfies $\|v\|_{L^6\left(S_{\deg(\overline{\vv})}\right)}^6\geq\frac32\|\phi_1\|_{L^6(\R)}^6$, and the convergence to one such $v$ of the restriction of $u_{\lambda_n}$ to $\displaystyle \bigcup_{i=1}^{\deg(\overline{\vv})}[0,a_i^n]$ would then yield
\[
\liminf_{n\to\infty}\JJ_{6,\G_{{\lambda_n}}}(1)=\liminf_{n\to\infty}J_{6,1}(u_{\lambda_n},\G_{\lambda_n})\geq\liminf_{n\to\infty}\frac13\|u_{\lambda_n}\|_{L^6\left(\bigcup_{i=1}^{\deg(\overline{\vv})}[0,a_i^n]\right)}^6
\geq\frac12\|\phi_1\|_{L^6(\R)}^6>\JJ_{6,\R}(1), 
\]
violating \eqref{eq:estR1} (since $\JJ_{6,\R}(1)=\|\phi_1\|_{L^6(\R)}^6/3$). Hence, \eqref{eq:alpha2} is proved. 

\smallskip
{\em Step 3: convergence of action ground states to $\phi_1$.}
By \eqref{eq:maxu} and \eqref{eq:alpha2},  each local maximum point of any positive action ground state $u_\lambda\in\NN_{6,1}(\G_\lambda)$ belongs to the interior of some edge as soon as $\lambda$ is sufficiently large.

For each of the local maximum points $\overline{x}\in\G_\lambda$ of $u_\lambda$, consider now the connected component of the superlevel set $\left\{x\in\G_\lambda\,:\, u_\lambda>\alpha_2\right\}$ containing $\overline{x}$. By \eqref{eq:maxu} and \eqref{eq:alpha2}, for each $\overline{x}$ this set is an interval fully contained into one edge of $\G_\lambda$ and, since $u_\lambda$ satisfies \eqref{eq:NLSul}, the restriction of $u_\lambda$ to each of these intervals is symmetric with respect to the middle point and monotone decreasing from the middle point to the endpoints of the interval. Observe furthermore that, again by \eqref{eq:NLSul} and the fact that the $L^\infty$ norm of $u_\lambda$ is bounded uniformly on $\lambda$, standard phase-plane analysis guarantees that, as $\lambda\to\infty$,  the length of each of these intervals diverges uniformly on the set of action ground states in $\NN_{6,1}(\G_\lambda)$. Now, among all these intervals (that, for fixed $u_\lambda$, are in finite number because the local maximum points of $u_\lambda$ are isolated), let $\overline{A}:=\left[-a_\lambda,a_\lambda\right]$ be the one on which the $L^6$ norm of the restriction of $u_\lambda$ is minimal. Note that
\[
u_\lambda'(-a_\lambda)\to0\qquad\text{as }\lambda\to0
\]
uniformly on the set of action ground states in $\NN_{6,1}(\G_\lambda)$. Indeed, if we assume by contradiction that this were not true, then there would exist numbers $\lambda_n\to\infty$ and action ground states $u_{\lambda_n}\in\NN_{6,1}(\G_{\lambda_n})$ for which $|u_{\lambda_n}'(-a_{\lambda_n})|\geq K$ for every $n$, for some constant $K>0$ independent of $n$. Since $u_{\lambda_n}(-a_{\lambda_n})=\alpha_2$,  \eqref{eq:alpha2} and standard phase-plane analysis would then imply that $u_{\lambda_n}$ changes sign on $\G_{\lambda_n}$, contradicting \eqref{eq:NLSul}. 

All in all, the restriction of $u_\lambda$ to $[-a_\lambda,a_\lambda]$ is a symmetric function, monotone decreasing on $[0,a_\lambda]$, satisfying
\[
\begin{cases}
	u_\lambda''+u_\lambda^5=u_\lambda & \text{on }[-a_\lambda,a_\lambda]\\
	u_\lambda>0,\quad \|u_\lambda\|_\infty=u_\lambda(0)\geq1 
\end{cases}
\]
and such that $u_\lambda(-a_\lambda)\to0, u_\lambda'(-a_\lambda)\to0$ as $\lambda\to\infty$, so that by continuity in the phase-plane it converges in $C_{\text{\normalfont loc}}^2$ to the soliton $\phi_1$ on $\R$.  

Observe that this convergence to $\phi_1$ implies that, as soon as $\lambda$ is large enough, any action ground state $u_\lambda\in\NN_{6,1}(\G_\lambda)$ has a unique local maximum point (that in $\overline{A}$). Indeed, if this were not the case, there would exist $\lambda_n\to\infty$ and $u_n\in\NN_{6,1}(\G_{\lambda_n})$ such that $u_n$ has at least two local maximum points on $\G_{\lambda_n}$. This would imply that the superlevel set $\left\{x\in\G_{\lambda_n}\,:\,u_n>\alpha_2\right\}$ has at least two connected components, and by the definition of $\overline{A}$ this would give $\JJ_{6,\G_{{\lambda_n}}}(1)=J_{6,\lambda_n}(u_n,\G_{\lambda_n})=\frac13\|u_n\|_{L^6(\G_{\lambda_n})}^6\geq \frac23\|u_n\|_{L^6(\overline{A})}^6\to\frac23\|\phi_1\|_{L^6(\R)}^6=2\JJ_{6,\R}(1)$ as $n\to\infty$, which is impossible by \eqref{eq:estR1}.

The uniqueness of the local maximum point of $u_\lambda$ ensures that 
\begin{equation}
\label{eq:max}
\max_{\G_\lambda\setminus\overline{A}}u_\lambda=\max_{\partial\overline{A}}u_\lambda\,,
\end{equation}
because if this were not the case $u_\lambda$ would have at least another local maximum point in $\G_\lambda\setminus\overline{A}$. By \eqref{eq:maxu}, \eqref{eq:alpha2} and the fact that $u_\lambda\equiv\alpha_2$ on $\partial\overline{A}$, it then follows that $\partial\overline{A}$ contains at least one vertex $\overline\vv$ of $\G_\lambda$, where $u_\lambda(\overline\vv)=\alpha_2$. Moreover, denoting by $(e_i)_{i=1}^{\deg(\overline\vv)}$ the edges of $\G_\lambda$ incident at $\overline\vv$, with $e_1$ being the unique one containing $\overline{A}$, there results 
\begin{equation}
\label{eq:der}
\begin{split}
	\frac{du_\lambda}{dx_{e_1}}(\overline\vv)>0\quad\text{and}\quad\frac{du_\lambda}{dx_{e_i}}(\overline\vv)<0\qquad\forall i=2,\dots,\deg(\overline\vv)\,.
\end{split}
\end{equation}
The first inequality follows by the monotonicity of $u_\lambda$ on $\overline{A}$, whereas the second one is a consequence of \eqref{eq:alpha2} and \eqref{eq:max} (that would be trivially violated if $u_\lambda$ were increasing at $\overline\vv$ along any edge incident at it other than $e_1$).

Now, for every $i=1,\dots, \deg(\overline\vv)$, let $x_i\in e_i$ be such that $\displaystyle \min_{e_i}u_\lambda=u_\lambda(x_i)$ (such $x_i$ is unique by the uniqueness of the local maximum point of $u_\lambda$ on $\G_\lambda$). Then, 
\begin{equation}
\label{eq:dist}
d_{e_i}(\overline\vv,x_i)\geq\frac{|e_i|}2\qquad\forall i=1,\dots,\deg(\overline\vv)\,,
\end{equation}
where $d_{e_i}(\overline\vv,x_i)$ denotes the distance between $\overline\vv$ and $x_i$ along $e_i$. Indeed, if $i\geq2$, since $u_\lambda$ satisfies \eqref{eq:NLSul}, by \eqref{eq:der} it is monotone decreasing on $e_i$ from $\overline\vv$ to $x_i$ and then monotone increasing from $x_i$ to the other vertex $\vv_i$ of $e_i$ (if this does not coincide with $x_i$). Therefore, if \eqref{eq:dist} were false for any $i\geq2$,  by the periodicity of the orbits in the phase-plane associated to \eqref{eq:NLSul} we would have $u_\lambda(\vv_i)>u_\lambda(\overline\vv)=\alpha_2$, which is impossible by \eqref{eq:max}. The same argument works when $i=1$, recalling that the restriction of $u_\lambda$ to $e_1$ is increasing from $\overline\vv$ to a local maximum point, decreasing from this maximum point to $x_1$ and then increasing again to the end of $e_1$ if $x_1\neq\vv_1$.

\smallskip
{\em Step 4: exponential decay of action ground states.} With the notation of the previous step, we now prove that there exist constants $c_1,c_2>0$, independent of $\lambda$, such that
\begin{equation}
\label{eq:expdec}
\max_{1\leq i \leq \deg(\overline\vv)}u_\lambda(x_i)\leq c_1e^{-c_2 \sqrt{\lambda}}\qquad\text{as }\lambda\to\infty
\end{equation}
uniformly on the set of action ground states in $\NN_{6,1}(\G_\lambda)$. To this end, fix $\varepsilon\in(0,1/4)$ and $R=2\log(C_0/\varepsilon)$, where $C_0>0$ is chosen such that $\phi_1(x)\leq C_0e^{-x}$ on $\R^+$ (such a $C_0$ exists by \eqref{eq:phi1}). By the $C_{\text{\normalfont loc}}^2$ convergence of $u_\lambda$ to $\phi_1$ on $\overline{A}$ (and recalling the identification of $\overline A$ with the interval $(-a_\lambda,a_\lambda)$), for sufficiently large $\lambda$ we have $R<a_\lambda$ and
\begin{equation}
\label{eq:eps1}
0<u_\lambda(x)\leq \|u_\lambda-\phi_1\|_{C^2(0,R)}+\phi_\lambda(x)\leq 2\varepsilon\qquad\forall x\in\left[\frac R2,R\right].
\end{equation}
With no loss of generality, let the vertex $\overline\vv\in\partial\overline A$ correspond to the endpoint $-a_\lambda$ of the interval $(-a_\lambda,a_\lambda)$. By definition, the point $x_1$ then belongs to the portion of the edge $e_1$ on the right of the point $\partial\overline A\setminus\left\{\overline\vv\right\}$, so that the part of $e_1$ starting at $\partial\overline A\setminus\left\{\overline\vv\right\}$ and ending at $x_1$ can be identified with the interval $[a_\lambda, b_{1,\lambda}]$, for some $b_{1,\lambda}\geq a_\lambda$. Hence, we can identify the interval $\displaystyle\left[\frac R2,b_{1,\lambda}\right]$ with the portion of $e_1$ starting at the point of $\overline A$ corresponding to $\frac R2$ and ending at $x_1$. By the monotonicity of $u_\lambda$ on $e_1$, \eqref{eq:eps1} ensures that
\begin{equation}
	\label{eq:eps2}
0<u_\lambda(x)\leq 2\varepsilon\qquad\forall x\in\left[\frac R2,b_{1,\lambda}\right]\cap e_1\,.
\end{equation}
Similarly, identifying for every $i\geq2$ the interval $[0,b_{i,\lambda}]$ with the portion of $e_i$ starting at $\overline\vv$ and ending at $x_i$, by \eqref{eq:eps1} and the monotonicity of $u_\lambda$ on $e_i$ we obtain
\[
0<u_\lambda(x)\leq2\varepsilon\qquad\forall x\in[0,b_{i,\lambda}]\cap e_i,\,\forall i\geq2\,.
\]
Notice that, by \eqref{eq:dist} and the definition of $\G_\lambda$, 
\[
\min_{1\leq i\leq\deg(\overline{\vv})}b_{i,\lambda}>\frac13\min_{1\leq i\leq\deg(\overline\vv)}|e_i|\geq \frac13\min_{e\in\E_\G}|e|\sqrt\lambda\,.
\]
In particular, there exists a constant $C>0$, independent of $\lambda$, such that, for every $i=1,\dots,\deg(\overline\vv)$, the restriction of $u_\lambda$ to the portion of $e_i$ of length $C\sqrt\lambda$ ending at $x_i$ is monotone decreasing and satisfies $0<u_\lambda\leq 2\varepsilon$. 

Fix now any $i\in\left\{1,\dots,\deg(\overline\vv)\right\}$ and think of this restriction of $u_\lambda$ on the corresponding  portion of $e_i$ as a decreasing function on the interval $[0,C\sqrt\lambda]$. Letting $u$ be the action ground state in $\NN_{6,\lambda}(\G)$ associated to $u_\lambda$ by \eqref{eq:ul}, the corresponding restriction of $u$ verifies
\begin{equation}
\label{eq:u<eps}
0<u(x)=\lambda^\frac14 u_\lambda\big(\sqrt{\lambda}\,x\big)\leq 2\lambda^\frac14\varepsilon\qquad\forall x\in\left[0,C\right], 
\end{equation}
it is monotone decreasing and satisfies $u''+u^5=\lambda u$ on $[0,C]$, that together with $\varepsilon<1/4$ yield
\begin{equation}
\label{eq:inequ}
u''=\lambda u-u^5=\lambda u(1-u^4/\lambda)\geq\lambda u\left(1-(2\varepsilon)^4\right)>\frac{\lambda}2u\quad\text{on }\left[0,C\right].
\end{equation}
Let then $\psi$ be the solution of the boundary value problem
\[
\begin{cases}
\psi''=\frac{\lambda}{2}\psi & \text{on }\left[0,C\right]\\
\psi(0)=u(0), & \psi(C)=u(C)\,.
\end{cases}
\]
Setting $w:=\psi-u$ on $[0,C]$ and coupling with \eqref{eq:inequ} gives
\[
\begin{cases}
w''<\frac{\lambda}2 w & \text{on }\left[0,C\right]\\
w(0)=w(C)=0\,,
\end{cases}
\]
so that the maximum principle implies $\psi\geq u$ on $\left[0,C\right]$. Since $\psi$ is explicitly given by
\[
\psi(x)=\frac{u(0)}{e^{2C\sqrt{\frac\lambda2}}-1}\left(e^{\sqrt{\frac\lambda2}\left(2C-x\right)}-e^{\sqrt{\frac\lambda2}\,x}\right)+\frac{u(C)}{e^{2C\sqrt{\frac\lambda2}}-1}\left(e^{\sqrt{\frac\lambda2}(C+x)}-e^{\sqrt{\frac\lambda2}(C-x)}\right),
\]
the monotonicity of $u$ on $[0,C]$, and the fact that $u(0)=\lambda^\frac14u_\lambda(0)\leq\lambda^\frac14\|u_\lambda\|_\infty=O\left(\lambda^\frac14\right)$ as $\lambda\to\infty$, entails
\[
u(C)<u(C/2)\leq\psi(C/2)\leq 4u(0)e^{-\frac C2\sqrt{\frac\lambda2}}\leq c_1\lambda^\frac14 e^{-\frac C2\sqrt{\frac\lambda2}}
\]
as soon as $\lambda$ is large enough. By \eqref{eq:ul}, it follows that $u_\lambda(x_i)=u_\lambda\big(C\sqrt\lambda\big)\leq c_1 e^{-\frac C2\sqrt{\frac\lambda2}}$,  that is the desired estimate for each given $i$. As the above computations are all uniform on $i=1,\dots,\deg(\overline\vv)$ and on the set of action ground states $u$ in $\NN_{6,\lambda}(\G)$, this proves \eqref{eq:expdec}.

\smallskip
{\em Step 5: estimate from below of $\JJ_{6,\G_{{\lambda}}}(1)$ and conclusion of the proof of \eqref{eq:levGl}--\eqref{eq:massGl}.} 
To complete the proof of \eqref{eq:levGl}, in view of the upper bound \eqref{eq:estR1} it is enough to show that there exists $\beta>0$ independent of $\lambda$ such that
\begin{equation}
	\label{eq:JGllower}
	\JJ_{6,\G_{{\lambda}}}(1)\geq\JJ_{6,\R}(1)-O\left(e^{-\beta\sqrt{\lambda}}\right)\qquad\text{as }\lambda\to\infty\,.
\end{equation}
To this end, given an action ground state $u_\lambda$ in $\NN_{6,1}(\G_\lambda)$, we denote by $\widetilde{\G}_\lambda$ the graph obtained attaching at the point $x_{\overline i}$ of $\G_\lambda$ two new bounded edges $p_1,p_2$, each of length 1, with $\overline i\in\left\{1,\dots,\deg(\overline\vv)\right\}$ such that $\displaystyle u_\lambda(x_{\overline i})=\max_{1\leq i\leq\deg(\overline\vv)}u_\lambda(x_i)$. Consider then the function $v_\lambda:\widetilde{\G}_\lambda\to\R$ defined as
\[
v_\lambda(x):=\begin{cases}
u_\lambda(x) & \text{if }x\in\G_\lambda\\
u_\lambda(x_{\overline i})(1-x) & \text{if }x\in p_1\\
u_\lambda(x_{\overline i})(1-x) & \text{if }x\in p_2\,,
\end{cases}
\]
where with a slight abuse of notation we identify in the obvious way corresponding points in $\G_\lambda$ and $\widetilde{\G}_\lambda$, and both $p_1,p_2$ with the interval $[0,1]$ so that $0$ is associated to the point of $\G_\lambda$ the two additional edges are attached at. By construction, $v_\lambda\in H^1(\widetilde{\G}_\lambda)$ and
\begin{equation}
\label{eq:normvl}
\begin{split}
&\|v_\lambda\|_{L^2\left(\widetilde{\G}_\lambda\right)}^2=\|u_\lambda\|_{L^2\left(\G_\lambda\right)}^2+\frac23u_\lambda(x_{\overline i})^2\\
&\|v_\lambda\|_{L^6\left(\widetilde{\G}_\lambda\right)}^6=\|u_\lambda\|_{L^6\left(\G_\lambda\right)}^6+\frac27u_\lambda(x_{\overline i})^6\\
&\|v_\lambda'\|_{L^2\left(\widetilde{\G}_\lambda\right)}^2=\|u_\lambda'\|_{L^2\left(\G_\lambda\right)}^2+2u_\lambda(x_{\overline i})^2\,,
\end{split}
\end{equation}
so that by \eqref{eq:expdec} (and the fact that $\|u_\lambda\|_{L^6(\G_\lambda)}$ is bounded away from 0 by Step 3)
\begin{equation}
\label{eq:svl}
\sigma_{6,1}(v_\lambda)^6=\left(\frac{\|u_\lambda'\|_2^2+\|u_\lambda\|_2^2+\frac83u_\lambda(x_{\overline i})^2}{\|u_\lambda\|_6^6+\frac27u_\lambda(x_{\overline i})^6}\right)^\frac32=1+\frac4{\|u_\lambda\|_6^6}u_\lambda(x_{\overline i})^2+o\left(u_\lambda(x_{\overline i})^2\right)\qquad\text{as }\lambda\to\infty.
\end{equation}
Furthermore, observe that $u_\lambda$ (and thus $v_\lambda$) attains all the values in $(u_\lambda(x_{\overline i}),\|u_\lambda\|_\infty)$ at least twice on $\G_\lambda$. Indeed, all the values in $[\alpha_2,\|u_\lambda\|_\infty)$ are attained exactly twice on the edge $e_1$, and each value in $(u_\lambda(x_{\overline i}),\alpha_2)$ (if any) is attained at least once on each edge incident at $\overline\vv$ by definition of $\overline i$ (and therefore at least three times on $\G_\lambda$ since $\deg(\overline\vv)\geq3$). Since by construction $v_\lambda$ attains all the values in $[0,u_\lambda(x_{\overline i}))$ exactly twice on $p_1\cup p_2$, we then have that $v_\lambda$ attains all the values in $(0,\|v_\lambda\|_\infty)$ at least twice on $\widetilde{\G}_\lambda$. Hence, by the standard theory of rearrangements on metric graphs (see e.g. \cite[Section 2]{DDGS}), if $w_\lambda\in H^1(\R)$ denotes the symmetric rearrangement on the real line of $\sigma_{6,1}(v_\lambda)v_\lambda$, then $\sigma_{6,1}(w_\lambda)\leq 1$ (since $\sigma_{6,1}(v_\lambda)v_\lambda\in\NN_{6,1}(\widetilde{\G}_\lambda)$), in turn yielding
\begin{equation}
\label{eq:wl}
\JJ_{6,\R}(1)\leq J_{6,1}(\sigma_{6,1}(w_\lambda)w_\lambda,\R)=\frac13\sigma_{6,1}(w_\lambda)^6\|w_\lambda\|_{L^6(\R)}^6\leq \frac13\sigma_{6,1}(v_\lambda)^6\|v_\lambda\|_{L^6(\widetilde{\G}_\lambda)}^6
\end{equation}
because $\sigma_{6,1}(w_\lambda)w_\lambda\in\NN_{6,1}(\R)$ by definition and $\|w_\lambda\|_{L^6(\R)}=\|\sigma_{6,1}(v_\lambda)v_\lambda\|_{L^6(\widetilde{\G}_\lambda)}$ by the equimeasurability of rearrangements. As $\lambda\to\infty$, coupling with \eqref{eq:normvl} and \eqref{eq:svl} leads to
\[
\begin{split}
\JJ_{6,\R}(1)\leq&\, \frac13\left(1+\frac4{\|u_\lambda\|_6^6}u_\lambda(x_{\overline i})^2+o\left(u_\lambda(x_{\overline i})^2\right)\right)\left(\|u_\lambda\|_{6}^6+\frac27u_\lambda(x_{\overline i})^6\right)\\
=&\,\frac13\|u_\lambda\|_6^6+\frac43u_\lambda(x_{\overline i})^2+o\left(u_\lambda(x_{\overline i})^2\right)=\JJ_{6,\G_{{\lambda}}}(1)+\frac43u_\lambda(x_{\overline i})^2+o\left(u_\lambda(x_{\overline i})^2\right)\,,
\end{split}
\]
that together with \eqref{eq:expdec} yields \eqref{eq:JGllower}. 

To conclude, we are thus left to prove \eqref{eq:massGl}. Observe that, since in the limit for $\lambda\to\infty$ all inequalities in \eqref{eq:wl} become equalities, $(\sigma_{6,1}(w_\lambda)w_\lambda)_\lambda\subset\NN_{6,1}(\R)$ is a minimizing sequence for $\JJ_{6,\R}(1)$ with $\|\sigma_{6,1}(w_\lambda)w_\lambda\|_\infty=\sigma_{6,1}(w_\lambda)w_\lambda(0)$, and it thus converges strongly in $H^1(\R)$ to $\phi_1$ as $\lambda\to\infty$. In particular, $\|\sigma_{6,1}(w_\lambda)w_\lambda\|_{L^2(\R)}^2\to\|\phi_1\|_{L^2(\R)}^2=\mu_\R$ as $\lambda\to\infty$. Furthermore, \eqref{eq:svl}, the equalities in \eqref{eq:wl}, and the equimeasurability of rearrangements give
\[
\begin{split}
\lim_{\lambda\to\infty}\|v_\lambda\|_{L^6(\widetilde{\G}_\lambda)}^6&\,=\lim_{\lambda\to\infty}\sigma_{6,1}(v_\lambda)^6\|v_\lambda\|_{L^6(\widetilde{\G}_\lambda)}^6=\lim_{\lambda\to\infty}\sigma_{6,1}(w_\lambda)^6\|w_\lambda\|_{L^6(\R)}^6\\
&\,=\lim_{\lambda\to\infty}\sigma_{6,1}(w_\lambda)^6\sigma_{6,1}(v_\lambda)^6\|v_\lambda\|_{L^6(\widetilde{\G}_\lambda)}^6=\lim_{\lambda\to\infty}\sigma_{6,1}(w_\lambda)^6\|v_\lambda\|_{L^6(\widetilde{\G}_\lambda)}^6\,,
\end{split}
\] 
in turn implying $\displaystyle\lim_{\lambda\to\infty}\sigma_{6,1}(w_\lambda)=1$, as $\displaystyle\lim_{\lambda\to\infty}\|v_\lambda\|_{L^6(\widetilde{\G}_\lambda)}^6>0$ again by \eqref{eq:wl}. Hence, exploiting once more the equimeasurability of $w_\lambda$ and $\sigma_{6,1}(v_\lambda)v_\lambda$, together with \eqref{eq:normvl}, we get
\[
\mu_\R=\lim_{\lambda\to\infty}\|\sigma_{6,1}(w_\lambda)w_\lambda\|_{L^2(\R)}^2=\lim_{\lambda\to\infty}\|w_\lambda\|_{L^2(\R)}^2=\lim_{\lambda\to\infty}\sigma_{6,1}(v_\lambda)^2\|v_\lambda\|_{L^2(\widetilde{\G}_\lambda)}^2=\lim_{\lambda\to\infty}\|u_\lambda\|_{L^2(\G_\lambda)}^2\,,
\]
i.e. \eqref{eq:massGl}.

\smallskip
The proof of \eqref{eq:critmasR+} and \eqref{eq:JGJR+} for graphs with at least one pendant is completely analogous, one simply needs to replace the levels on $\R$ with the corresponding ones on $\R^+$, to note that Steps 2-3 as above will prove convergence of action ground states on $\G_\lambda$ to the half-soliton (i.e. the restriction of $\phi_1$ to $\R^+$) on one pendant of the graph, and to make use of monotone rearrangements in place of symmetric ones in Step 5. 
\end{proof}

\section{Proof of Theorem \ref{thm:compact}}
\label{sec:comp}

In this section we complete the proof of our main non-uniqueness result on compact graphs, Theorem \ref{thm:compact}.
We begin with the next preliminary lemma.
\begin{lemma}
	\label{lem:l1>l2}
	Let $\G$ be a compact graph. If there are numbers $0<\lambda_1<\lambda_2$ such that $M_{6,\G}^-(\lambda_1)>M_{6,\G}^+(\lambda_2)$, then there exists $p_\G>2$ such that $Z_{p,\G}\neq\emptyset$ for every $p\in[p_\G,6)$.
\end{lemma}
\begin{proof}
	Set $\delta:=M_{6,\G}^-(\lambda_1)-M_{6,\G}^+(\lambda_2)>0$. By Remark \ref{rem:comp}, there exists $\varepsilon>0$ such that
	\begin{equation}
	\label{eq:absM}
	\inf_{p\in(6-\varepsilon,6)}M_{p,\G}^-(\lambda_1)\geq M_{6,\G}^-(\lambda_1)-\frac{\delta}3>M_{6,\G}^+(\lambda_2)+\frac{\delta}3\geq\sup_{p\in(6-\varepsilon,6)}M_{p,\G}^+(\lambda_2)\,.
	\end{equation}
	To prove the lemma, assume then by contradiction that there exists $\overline{p}\in(6-\varepsilon,6)$ such that $Z_{\overline{p},\G}=\emptyset$. By Proposition \ref{prop:Ecomp} this means that the corresponding map $\Lambda_{\overline{p},\G}:\R^+\to\R^+$ is strictly increasing and surjective, so that for every $\lambda>0$ there exists a unique value $\mu(\lambda)>0$ and a ground state $u$ of $E_{\overline{p}}(\cdot,\G)$ with mass $\mu(\lambda)$ satisfying $\LL_{\overline{p}}(u,\G)=\lambda$. By Proposition \ref{prop:AEgs}, this implies that all action ground states in $\NN_{\overline{p},\lambda}(\G)$ have mass equal to $\mu(\lambda)$, namely $M_{\overline{p},\G}^+(\lambda)=\mu(\lambda)=M_{\overline{p},\G}^-(\lambda)$ for every $\lambda>0$.
	
	In particular, there are $\mu_1,\mu_2>0$ and $u_1\in H_{\mu_1}^1(\G)$, $u_2\in H_{\mu_2}^1(\G)$ such that $u_1$ is a ground state of $E_{\overline{p}}(\cdot,\G)$ with mass $\mu_1$ and $\LL_{\overline{p}}(u_1,\G)=\Lambda_{\overline{p},\G}(\mu_1)=\lambda_1$, and $u_2$ is a ground state of $E_{\overline{p}}(\cdot,\G)$ with mass $\mu_2$ and $\LL_{\overline{p}}(u_2,\G)=\Lambda_{\overline{p},\G}(\mu_2)=\lambda_2$. On the one hand, the strict monotonicity of $\Lambda_{\overline{p},\G}$ yields $\mu_1<\mu_2$. On the other hand, since $u_1$ and $u_2$ are also action ground states in $\NN_{\overline{p},\lambda_1}(\G)$ and $\NN_{\overline{p},\lambda_2}(\G)$ respectively, there results $M_{\overline{p},\G}^\pm(\lambda_1)=\mu_1$ and $M_{\overline{p},\G}^\pm(\lambda_2)=\mu_2$, that combined with \eqref{eq:absM} entails $\mu_1>\mu_2$. As this provides the desired contradiction, we conclude.
\end{proof}
\begin{proof}[Proof of Theorem \ref{thm:compact}]
	Let us first consider compact graphs $\G$ with no pendant. By Lemma \ref{lem:l1>l2}, to prove Theorem \ref{thm:compact} it is enough to verify that there exist $\lambda_1<\lambda_2$ such that $M_{6,\G}^-(\lambda_1)>M_{6,\G}^+(\lambda_2)$. To this end, recall that, when $p=6$ and $\G$ has no pendant, energy ground states exist for every $\mu\leq\mu_\R$ (see Proposition \ref{prop:Ecomp}). In particular, there exists a ground state $\overline{u}\in H_{\mu_\R}^1(\G)$ of $E_6(\cdot,\G)$ with mass $\mu_\R$ and $\LL_6(\overline{u},\G)=\overline{\lambda}$, for some $\overline\lambda>0$. By Proposition \ref{prop:AEgs}, it then follows that all action ground states in $\NN_{6,\overline\lambda}(\G)$ have mass $\mu_\R$, i.e. $M_{6,\G}^\pm(\overline{\lambda})=\mu_\R$. Furthermore, by Proposition \ref{prop:est6}, $\displaystyle\lim_{\lambda\to\infty}M_{6,\G}^\pm(\lambda)=\mu_\R$. Hence, to complete the proof we are left to show that there exists $\widetilde{\lambda}>\overline{\lambda}$ such that $M_{6,\G}^-(\widetilde{\lambda})=M_{6,\G}^+(\widetilde{\lambda})\neq\mu_\R$.
	
	Assume by contradiction that this is false, namely that $M_{6,\G}^\pm(\lambda)=\mu_\R$ for every $\lambda\in\left[\overline{\lambda},\infty\right)\setminus \widetilde{Z}_{6,\G}$ (the set $\widetilde{Z}_{6,\G}$ is defined in Proposition \ref{prop:Acomp}). Since $\JJ_{6,\G}$ is locally Lipschitz continuous and differentiable out of the at most countable set $\widetilde{Z}_{6,\G}$, Proposition \ref{prop:Acomp} yields
	\[
	\JJ_{6,\G}(\lambda)=\JJ_{6,\G}(\overline{\lambda})+\int_{\overline{\lambda}}^{\lambda}\JJ_{6,\G}'(s)\,ds=\JJ_{6,\G}(\overline{\lambda})+\frac{\mu_\R}2\left(\lambda-\overline{\lambda}\right)
	\]
	for every $\lambda\in\left[\,\overline{\lambda},\infty\right)\setminus\widetilde{Z}_{6,\G}$, and thus everywhere on $\left[\,\overline{\lambda},\infty\right)$ by the continuity of $\JJ_{6,\G}$. By Remark \ref{rem:lev6R} and Proposition \ref{prop:est6}, it then follows
	\[
	\left|\JJ_{6,\G}(\lambda)-\JJ_{6,\R}(\lambda)\right|=\left|\JJ_{6,\G}(\overline{\lambda})-\frac{\mu_\R}2\overline{\lambda}\right|=o\left(e^{-\gamma\sqrt{\lambda}}\right)\qquad\text{as }\lambda\to\infty
	\]
	that is $\displaystyle \JJ_{6,\G}(\overline{\lambda})-\frac{\mu_\R}2\overline{\lambda}=0$. Since $\LL_6(\overline{u},\G)=\overline{\lambda}$ and $\JJ_{6,\G}(\overline{\lambda})=J_{6,\overline{\lambda}}(\overline{u},\G)$ for the ground state $\overline{u}\in H_{\mu_\R}^1(\G)$ of $E_6(\cdot,\G)$ with mass $\mu_\R$, this entails $\EE_{6,\G}(\mu_\R)=E_6(\overline{u},\G)=\JJ_{6,\G}(\overline{\lambda})-\frac{\mu_\R}2\overline{\lambda}=0$, contradicting the strict negativity of $\EE_{6,\G}(\mu_\R)$ ensured by Proposition \ref{prop:Ecomp}.
	
	\smallskip
	The proof for graphs with at least one pendant is identical, simply replacing $\mu_\R$ with $\mu_{\R^+}$.
\end{proof}

\begin{remark}
	\label{rem:Dircomp}
	As the previous proof shows, the existence of $\lambda_1,\lambda_2$ as in Lemma \ref{lem:l1>l2} on compact graphs (and thus the validity of Theorem \ref{thm:compact}) is a consequence of the fact that $L^2$-critical energy ground states exist if and only if the mass is smaller than or equal to a threshold value. Observe that this is not what happens in the $L^2$-critical regime in similar settings, as e.g. on bounded domains of $\R^N$ with homogeneous Dirichlet conditions at the boundary. In this latter case, ground states exist if and only if the mass is strictly smaller than a threshold (see e.g. \cite[Theorem 1.5(ii)]{NTV}). Hence, since ground states at the threshold are crucial in the proof of Theorem \ref{thm:compact}, we do not expect to observe the same non-uniqueness phenomenon in the context of bounded domains with homogeneous Dirichlet boundary conditions.
\end{remark}

\section{Noncompact graphs with finitely many edges: proof of Theorem \ref{thm:noncomp}}
\label{sec:Ghalf}

This section focuses on noncompact graphs with finitely many edges and shows how to prove Theorem \ref{thm:noncomp}. Recall that this theorem takes into account a limited set of graphs: the tadpole (Figure \ref{fig-NlT}(\textsc{\tiny A})), the signpost (Figure \ref{fig:GShalf}(\textsc{\tiny A})), the 2-fork (Figure \ref{fig:GShalf}(\textsc{\tiny B})) and the 3-fork (Figure \ref{fig:GShalf}(\textsc{\tiny C})). Thanks to their high specificity,  it is thus not difficult to adapt the argument in the proof of Theorem \ref{thm:compact} to obtain Theorem \ref{thm:noncomp}. However, and even though the main differences arise with respect to the discussion in Section 4, for the sake of completeness we now comment upon all the results of Sections \ref{sec:prel}--\ref{sec:est6}--\ref{sec:comp}, highlighting the modifications needed for our analysis to extend.

Proposition \ref{prop:Ecomp}(i) does not change (see \cite{AST2} for further details on $L^2$-subcritical energy ground states on the graphs considered here), whereas Proposition \ref{prop:Ecomp}(ii) is quite different. Indeed, \cite[Theorems 3.1--3.3--3.4]{ASTcmp} show that when $p=6$
\begin{itemize}
	\item if $\G$ is either the tadpole or the signpost graph, then 
	\[
	\EE_{6,\G}(\mu)\begin{cases}
		=0 & \text{if }\mu\leq\mu_\G\\
		<0 & \text{if }\mu_\G<\mu\leq\mu_\R\\
		=-\infty & \text{if }\mu>\mu_\R\,.
	\end{cases}
	\]
	Moreover, if $\G$ is the tadpole, then $\mu_\G=\mu_{\R^+}$ and $\EE_{6,\G}(\mu)$ is attained if and only if $\mu\in(\mu_{\R^+},\mu_\R]$, whereas if $\G$ is the signpost, then $\mu_{\R^+}<\mu_\G<\mu_\R$ and $\EE_{6,\G}(\mu)$ is attained if and only if $\mu\in[\mu_\G,\mu_\R]$;
	
	\item if $\G$ is either the 2- or the 3-fork graph, then
	\[
	\EE_{6,\G}(\mu)=\begin{cases}
		0 & \text{if }\mu\leq\mu_{\R^+}\\
		-\infty & \text{if }\mu>\mu_{\R^+}
	\end{cases}
	\]
	and it is never attained.
\end{itemize}
Finally, all properties of $\Lambda_{p,\G}^\pm$ listed in Proposition \ref{prop:Ecomp} are unchanged. The proof is almost the same, the only difference arising with the asymptotic formulas \eqref{eq:lim}, that for noncompact graphs with finitely many edges directly follow by \eqref{eq:LE} and the a priori estimates on energy ground states given in \cite[Lemma 2.6]{AST2}.

Proposition \ref{prop:Acomp} is identical, but minor comments should be given on how to prove it. The positivity of $\JJ_{p,\G}(\lambda)$ and the fact that it is attained for every $\lambda>0$ has been proved e.g. in \cite{DDGST}, from where it is also easy to see that $\JJ_{p,\G}(\lambda)\to0$ as $\lambda\to0^+$. That $\JJ_{p,\G}(\lambda)=0$ for every $\lambda\leq0$ follows by \eqref{eq:JonN}, the inequality $\JJ_{p,\G}(\lambda)\leq \JJ_{p,\R}(\lambda)$ (that holds for every graph with at least one half-line), and the fact $\JJ_{p,\R}(\lambda)=0$ for non-positive $\lambda$. As for the properties of $M_{p,\G}^\pm$, there is only a slight difference when proving that they are attained for every $\lambda>0$. As Step 3 of the proof of Proposition \ref{prop:Acomp} makes evident, this boils down to show that any sequence $(u_n)_n$ of action ground states in $\NN_{p,\lambda}(\G)$ converges strongly in $H^1(\G)$ (up to subsequences) to another action ground state $u\in\NN_{p,\lambda}(\G)$. Since the non-compactness of $\G$ is given here only by its half-lines, it is not difficult to see that it is then enough to show that the restriction of $u_n$ to each half-line $\mathcal{H}$ of $\G$ converges strongly. Once this is done, the compactness of the rest of $\G$ allows to recover the desired convergence of action ground states to action ground states. To prove the strong convergence on any half-line $\mathcal{H}$ of $\G$, recall that, if $u_n$ is a positive action ground state in $\NN_{p,\lambda}(\G)$, its restriction to $\mathcal{H}$ coincides with the restriction to a suitable interval $[a_n,\infty)$ of the unique positive $H^1$ solution $v$ of the problem
\[
\begin{cases}
	v''+v^{p-1}=\lambda v & \text{on }\R\\
	v(0)=\|v\|_{\infty}\,.
\end{cases}
\]
Since it is well-known that such a solution is symmetric, monotone decreasing on $[0,\infty)$ and satisfies $J_{p,\lambda}(v,\R)=\JJ_{p,\R}(\lambda)$, the strong convergence of $u_n$ on $\mathcal{H}$ follows if we prove that $\displaystyle \liminf_{n\to\infty} a_n>-\infty$. This latter inequality is guaranteed by the fact that $J_{p,\lambda}(u_n,\G)=\JJ_{p,\G}(\lambda)<\JJ_{p,\R}(\lambda)$ for every $\lambda>0$ and every noncompact graph we are dealing with (see e.g. \cite[Section 4]{DDGST}). Indeed, if along a subsequence (not relabeled) it were $\displaystyle \lim_{n\to\infty} a_n=-\infty$, we would have $J_{p,\lambda}(u_n,\G)\geq J_{p,\lambda}(u_n,\mathcal{H})\to J_{p,\lambda}(v,\R)=\JJ_{p,\R}(\lambda)$ as $n\to\infty$, i.e. a contradiction. 

Observe that, with the same consideration, one retrieves also the content of Remark \ref{rem:comp}. Indeed, for any fixed $\lambda>0$ and $p_n\to \overline{p}\in(2,\infty)$ as $n\to\infty$, arguing as we just did we also obtain the strong convergence of any sequence of action ground states $u_n\in\NN_{p_n,\lambda}(\G)$ on each half-line of $\G$. Since this implies in particular that the $L^\infty$ norm of $u_n$ on any half-line of $\G$ is attained at a uniformly bounded distance from the compact core of the graph (i.e. the set of all its bounded edges), we can rely again on the Dominated Convergence Theorem to recover \eqref{eq:ppbar}, and the rest of Remark \ref{rem:comp} then works exactly as in Section \ref{sec:prel}.

Proposition \ref{prop:AEgs} does not require any discussion, as it extends with no change to noncompact graphs with finitely many edges, and the same is true for the whole analysis of Section \ref{sec:est6} (no modification is needed in the proof of Proposition \ref{prop:est6}, since all graphs we are considering have finitely many vertices and edges and any function in $H^1(\G)$ tends to zero along the half-lines of $\G$).

Since the proof of Lemma \ref{lem:l1>l2} is based on Remark \ref{rem:comp}, Proposition \ref{prop:AEgs} and the $L^2$-subcritical part of Proposition \ref{prop:Ecomp}, it extends verbatim to the noncompact graphs we are taking into account here. Hence, to complete the proof of Theorem \ref{thm:noncomp}, we are left to show that, for each of the graphs covered by the theorem, there exist $0<\lambda_1<\lambda_2$ such that $M_{6,\G}^-(\lambda_1)>M_{6,\G}^+(\lambda_2)$.  

When $\G$ is either the tadpole or the signpost, the proof of this fact is identical to the one developed in Section \ref{sec:comp}. Indeed, as we said before, on both graphs there exists an $L^2$-critical energy ground state $\overline{u}$ with mass $\mu_\R$ such that $E_6(\overline{u},\G)<0$. Moreover, as Proposition \ref{prop:est6} applies to both graphs and they have no pendant, $\displaystyle\lim_{\lambda\to\infty}M_{6,\G}^\pm(\lambda)=\mu_\R$. Hence, setting again $\overline\lambda:=\LL_6(\overline{u},\G)$, the same computations at the end of Section \ref{sec:comp} concludes the proof of Theorem \ref{thm:noncomp} in the case of the tadpole and the signpost.

\begin{remark}
	Equivalently, the existence of $\lambda_1,\lambda_2$ as above on the tadpole graph can be deduced directly by the results in \cite{NP}.
\end{remark}

It remains to deal with the 2- and 3-fork graphs. To this end, we need the following lemma.

\begin{lemma}
	\label{lem:pend}
	Let $\G$ be a noncompact graph with finitely many edges and exactly one half-line. Assume that action ground states in $\NN_{6,\lambda}(\G)$ exist for every $\lambda>0$. Then for every $\varepsilon>0$ there exists $\widehat\lambda>0$ (depending on $\varepsilon$) such that
	\begin{equation}
	\label{eq:mu1half}
	\mu_{\R^+}-\varepsilon\leq M_{6,\G}^-(\lambda)\leq M_{6,\G}^+(\lambda)\leq \mu_{\R^+}+\varepsilon\qquad\forall \lambda\leq\widehat\lambda\,.
	\end{equation}
\end{lemma}
\begin{proof}
	Recall that, if $u\in\NN_{6,\lambda}(\G)$ is an action ground state of $J_{6,\lambda}(\cdot,\G)$, then $\displaystyle u_\lambda(x):=\lambda^{-\frac14}u\big(x/\sqrt\lambda\big)$
	belongs to $\NN_{6,1}(\G_\lambda)$ and is an action ground state of $J_{6,1}(\cdot,\G_\lambda)$, with $\G_\lambda:=\sqrt\lambda\G$. Since $\|u\|_{L^2(\G)}=\|u_\lambda\|_{L^2(\G_\lambda)}$ for every $\lambda$, to prove \eqref{eq:mu1half} it is enough to show that
	\begin{equation}
	\label{eq:M61half}
	\lim_{\lambda\to0^+}M_{6,\G_\lambda}^\pm(1)=\mu_{\R^+}\,.
	\end{equation}
	Recall first that, by \cite[Proposition 4.1]{DDGST},
	\begin{equation}
	\label{eq:estbel}
	\JJ_{6,\R^+}(1)\leq \JJ_{6,\G_{{\lambda}}}(1)\qquad\forall\lambda>0\,.
	\end{equation}
	Moreover, consider the function $v:\G_\lambda\to\R$ defined by
	\[
	v(x):=\begin{cases}
	\phi_1(x) & \text{if }x\in\HH\\
	\phi_1(0) & \text{if }x\in\K_\lambda\,,
	\end{cases}
	\]
	where $\phi_1$ is the soliton on the real line given by \eqref{eq:phi1}, $\HH$ is the unique half-line of $\G_\lambda$ (identified as usual with $\R^+$), and $\K_\lambda:=\sqrt\lambda\K$ is the compact core of $\G_\lambda$ (where $\K$ denotes the compact core of $\G$). As $\lambda\to0^+$, since $|\K_\lambda|\to0$ it is evident that 
	\[
	\|v\|_{L^q(\G_\lambda)}\to\|\phi_1\|_{L^q(\R^+)},\qquad\|v'\|_{L^2(\G_\lambda)}\to\|\phi_1'\|_{L^2(\R^+)}\,,
	\]
	so that $\sigma_{6,1}(v)\to\sigma_{6,1}(\phi_1)=1$ and 
	\[
	\JJ_{6,\G_{{\lambda}}}(1)\leq J_{6,1}(\sigma_{6,1}(v)v,\G_\lambda)=\frac13\sigma_{6,1}(v)^6\|v\|_{L^6(\G_\lambda)}^6\to\frac13\|\phi_1\|_{L^6(\R^+)}^6=\JJ_{6,\R^+}(1)\qquad\text{as }\lambda\to0^+\,.
	\]
	Coupling with \eqref{eq:estbel} yields 
	\begin{equation}
	\label{eq:limR+}
	\lim_{\lambda\to0^+}\JJ_{6,\G_{{\lambda}}}(1)=\JJ_{6,\R^+}(1).
	\end{equation}
	Let then $u_\lambda\in\NN_{6,1}(\G_\lambda)$ be such that $\JJ_{6,\G_{{\lambda}}}(1)=J_{6,1}(u_\lambda,\G_\lambda)$, and let $u_\lambda^*\in H^1(\R^+)$ be its monotone rearrangement on $\R^+$. By standard properties of rearrangements on graphs (see e.g. \cite[Section 2]{DDGS}), $u_\lambda\in\NN_{6,1}(\G_\lambda)$ implies $\sigma_{6,1}(u_\lambda^*)\leq 1$, that together with the equimeasurability of rearrangements entails
	\[
	\JJ_{6,\R^+}(1)\leq J_{6,1}(\sigma_{6,1}(u_\lambda^*)u_\lambda^*,\R^+)=\frac13\sigma_{6,1}(u_\lambda^*)^6\|u_\lambda^*\|_{L^6(\R^+)}^6\leq\frac13\|u_\lambda\|_{L^6(\G_\lambda)}^6=\JJ_{6,\G_{{\lambda}}}(1)\,.
	\]
	By \eqref{eq:limR+}, the previous inequalities become equalities in the limit for $\lambda\to0^+$. This implies that $(\sigma_{6,1}(u_\lambda^*)u_\lambda^*)_\lambda$ is a minimizing sequence for $J_{6,1}(\cdot,\R^+)$ in $\NN_{6,1}(\R^+)$. As so, it converges strongly in $H^1(\R^+)$ to $\phi_1$, and in particular
	\[
	\|\sigma_{6,1}(u_\lambda^*)u_\lambda^*\|_{L^2(\R^+)}^2\to\mu_{\R^+}\qquad\text{as }\lambda\to0^+\,.
	\]
	Since the above identities also show that $\sigma_{6,1}(u_\lambda^*)\to1$, exploiting again the equimeasurability of rearrangements yields $\|u_\lambda\|_{L^2(\G_\lambda)}^2\to\mu_{\R^+}$ as $\lambda\to0^+$, i.e. \eqref{eq:M61half}, and we conclude.
\end{proof}

Combining Lemma \ref{lem:pend} and Proposition \ref{prop:est6}(ii), we can now repeat the computations in the final part of Section \ref{sec:comp} to complete the proof of Theorem \ref{thm:noncomp} also for the 2- and 3-fork graphs. Indeed, since both graphs have exactly one half-line and at least one pendant, we obtain
\[
\lim_{\lambda\to0^+}M_{6,\G}^\pm(\lambda)=\mu_{\R^+}=\lim_{\lambda\to\infty}M_{6,\G}^\pm(\lambda).
\]
In light of this, to show that there exist $0<\lambda_1<\lambda_2$ for which $M_{6,\G}^-(\lambda_1)>M_{6,\G}^+(\lambda_2)$, we are left to exhibit at least one value of $\lambda>0$ satisfying $M_{6,\G}^-(\lambda)=M_{6,\G}^+(\lambda)\neq\mu_{\R^+}$. If this were not the case, arguing analogously to Section \ref{sec:comp} (and recalling that $\JJ_{6,\G}(0)=0$) we would have
\[
\JJ_{6,\G}(\lambda)=\int_0^\lambda\JJ_{6,\G}'(s)\,ds= \frac{\mu_{\R^+}}2\lambda
\]
for every $\lambda\in\R^+\setminus\widetilde{Z}_{6,\G}$, and thus for every $\lambda\in\R^+$ by the continuity of $\JJ_{6,\G}$. By Remark \ref{rem:lev6R}, this would mean $\JJ_{6,\G}\equiv\JJ_{6,\R^+}$, providing the contradiction we seek, since in fact it holds $\JJ_{6,\G}(\lambda)>\JJ_{6,\R^+}(\lambda)$ for every $\lambda>0$. This latter inequality can be easily deduced noting that, if equality holds, then the monotone rearrangement on $\R^+$ of any action ground state in $\NN_{6,\lambda}(\G)$ must coincide with the half-soliton $\phi_1$, but this is impossible due to the strict positivity (up to sign) of action ground states everywhere on $\G$.

\section{$\Z$-periodic graphs: proof of Theorem \ref{thm:per}}
\label{sec:Zper}

This final section is devoted to the proof of Theorem \ref{thm:per}, that reports on non-uniqueness of energy ground states on the ladder graph. Exactly as for Theorem \ref{thm:noncomp} on noncompact graphs with finitely many edges, the line of the proof of Theorem \ref{thm:per} is identical to that of Theorem \ref{thm:compact}. Hence, analogously to what we did in Section \ref{sec:Ghalf}, here we only highlight the main differences with respect to Sections \ref{sec:prel}--\ref{sec:est6}--\ref{sec:comp}. Nevertheless, this section is rather longer than the previous one, because recovering some of these results on the ladder graph requires much more work.

From now on, $\G$ will always denote the ladder graph (Figure \ref{fig:Zper}) with edges of length 1. In some cases, it will be convenient to think of $\G$ as embedded in $\R^2$ through the identification
\begin{equation}
\label{eq:GinR2}
\begin{split}
&\V_\G\simeq\bigcup_{k\in\Z}(k,0)\cup(k,1)\\
&\E_\G\simeq E_0\cup E_1 \cup E_2\,,
\end{split}
\end{equation}
where
\[
E_0=\bigcup_{k\in\Z}[k,k+1]\times\{0\},\qquad E_1=\bigcup_{k\in\Z}[k,k+1]\times\{1\},\qquad E_2=\bigcup_{k\in\Z}\{k\}\times[0,1]\,.
\]
For every $k\in\Z$, we will denote by $e_{0,k}$ the edge $[k,k+1]\times\{0\}$ in $E_0$, by $e_{1,k}$ the edge $[k,k+1]\times\{1\}$ in $E_1$, and by $e_{2,k}$ the edge $\{k\}\times[0,1]$ in $E_2$.

Let us then start discussing the extension to the ladder graph of the results in Section \ref{sec:prel}. Proposition \ref{prop:Ecomp}(i) does not change (see \cite[Theorem 1.1]{D19}). Conversely, Proposition \ref{prop:Ecomp}(ii) has to be modified according to \cite[Theorem 1.2]{D19}: the $L^2$-critical energy ground state level satisfies 
\[
\EE_{6,\G}(\mu)=\begin{cases}
	0 & \text{if }\mu\leq\mu_\R\\
	-\infty & \text{if }\mu>\mu_\R
\end{cases}
\]
and it is never attained. All properties of $\Lambda_{p,\G}^\pm$ are unchanged, and we only have to comment on how to prove \eqref{eq:lim}. Recall the Gagliardo--Nirenberg inequalities
\begin{equation*}
\|u\|_p^p\leq K_p\|u\|_2^{\frac p2+1}\|u'\|_2^{\frac p2-1},\qquad\forall p>2\,,
\end{equation*}
holding for every $u\in H^1(\G)$ and for suitable constants $K_p>0$ (see e.g. \cite{D19}). For every $p\in(2,6)$, coupling with the strict negativity of $\EE_{p,\G}$ shows that there exists $c_p>0$ (depending only on $p$) such that $\EE_{p,\G}(\mu)\geq-c_p\mu^{\frac{p+2}{6-p}}$ and $\|u\|_p^p\leq c_p\mu^{\frac{p+2}{6-p}}$ for every energy ground state $u$ in $H_\mu^1(\G)$. Since $(p+2)/(6-p)>1$ when $p\in(2,6)$, combining with \eqref{eq:LE} yields the first line of \eqref{eq:lim}. Similarly, the case $\lambda\to\infty$ follows by \eqref{eq:LE} and the estimate $\EE_{p,\G}(\mu)\leq -C_p\mu^{\frac{p+2}{6-p}}$, that for sufficiently large $\mu$ can be obtained e.g. considering functions in $H_\mu^1(\G)$ compactly supported on a single edge of the graph.  

The properties of $\JJ_{p,\G}(\lambda)$ for $\lambda>0$ given in Proposition \ref{prop:Acomp} have been proved on periodic graphs e.g. in \cite[Section 5]{DDGST}. That $\JJ_{p,\G}(\lambda)=0$ for every $\lambda<0$ is a consequence of the fact that, for any such $\lambda$, there exists a compact subset of $\G$ where the first eigenvalue of $-d^2/dx^2$ endowed with homogeneous Dirichlet conditions at the boundary (and standard Kirchhoff conditions at the vertices in the interior of the set) is exactly $-\lambda$, whereas the identity $\JJ_{p,\G}(0)=0$ can be proved arguing e.g. as in Step 1 of the proof of Lemma \ref{lem:GtoR} below. As for the properties of $M_{p,\G}^\pm$, nothing changes with respect to Section \ref{sec:prel} (in particular, one can show that they are attained for every $\lambda>0$ arguing exactly as in Part 1 of the proof of \cite[Theorem 1.7]{DDGST}). 

Contrary to Propositions \ref{prop:Ecomp}--\ref{prop:Acomp}--\ref{prop:AEgs} (this latter is always the same on any graph), the extension of the content of Remark \ref{rem:comp} to the ladder graph calls for a finer analysis. 
\begin{lemma}
	\label{lem:contp}
	For every $\overline{p}\in(2,\infty)$ and $\lambda>0$, there results $\displaystyle\lim_{p\to\overline{p}}\JJ_{p,\G}(\lambda)=\JJ_{\overline{p},\G}(\lambda)$.
\end{lemma}
\begin{proof}
	
	Let $\overline{p}\in(2,\infty)$ be fixed and, with no loss of generality, $\lambda=1$. To ease the notation, since the graph is always the same, in the rest of the proof we suppress the explicit dependence of all quantities on $\G$ and $\lambda$ (e.g. we write $\JJ_p$ in place of $\JJ_{p,\G}(1)$).
	
	Note first that, arguing exactly as in Remark \ref{rem:comp}, we obtain
	\begin{equation}
	\label{eq:limsup}
	\limsup_{p\to\overline{p}}\JJ_p\leq \JJ_{\overline{p}}\,.
	\end{equation}
	Hence, we are left to show that
	\begin{equation}
		\label{eq:liminf}
		\JJ_{\overline{p}}\leq\liminf_{p\to\overline{p}} \JJ_{p}.
	\end{equation}
	The proof of \eqref{eq:liminf} is a bit long, so we split it into three steps.
	
	\smallskip
	{\em Step 1.} Let $(p_n)_n$ be any sequence such that $p_n\to\overline{p}$ as $n\to\infty$, and $u_n\in\NN_{p_n}$ be such that $J_{p_n}(u_n)=\JJ_{p_n}$ for every $n$. As usual, we can take $u_n>0$ on $\G$. Moreover, exploiting \eqref{eq:GinR2} and the periodicity of $\G$, we can assume with no loss of generality that each $u_n$ attains its $L^\infty$ norm somewhere in $\K_0:=\displaystyle \bigcup_{i=0}^2 e_{i,0}$. 
	
	We claim that there exists a compact set $\K\subset\G$, independent of $n$, such that
	\begin{equation}
		\label{eq:compactA}
	\left\{x\in\G\,:\,u_n\geq1\right\}\subset \K\qquad\forall n\,.
	\end{equation}
	Observe first that, since $u_n$ solves \eqref{eq:nlse} with $p=p_n$ and $\lambda=1$, then 
	\begin{equation}
		\label{eq:un>1}
	u_n(\texttt{t})\geq1\qquad\forall n
	\end{equation}
	at any local maximum point $\texttt{t}\in\G$ of $u_n$, so that the superlevel set $\{u_n\geq1\}$ is non-empty. Furthermore, since $u_n$ attains its $L^\infty$ norm in $\K_0$  by assumption, we have $\{u_n\geq1\}\cap\displaystyle \K_0\neq\emptyset$.
	
	Assume by contradiction that \eqref{eq:compactA} is false, namely that (up to subsequences) there exists $\texttt{t}_n\in\G$ such that $u_n(\texttt{t}_n)\geq1$ for every $n$ and $d_\G(\texttt{t}_n,\K_0)\to\infty$ as $n\to\infty$ (where $d_\G(\texttt{t}_n,\K_0)$ is the distance in $\G$ between the point $\texttt{t}_n$ and the set $\K_0$). Relying on \eqref{eq:GinR2}, let $\texttt{t}_n=(x_n,y_n)\in\R^2$, so that (up to subsequences and with no loss of generality) $x_n\to\infty$ as $n\to\infty$.

	Since $(u_n)_n$ is bounded in $H^1(\G)$ by \eqref{eq:JonN} and \eqref{eq:limsup}, for every $\delta>0$ it must be
	\begin{equation}
		\label{eq:delta}
	\sup_n\#\left\{k\in\Z\cap[0,x_n]\,:\,\min_{e_{0,k}}u_n>\delta\right\}<\infty\,.
	\end{equation}
	Indeed, if this were not the case, there would exists $\overline\delta>0$ and a subsequence (not relabeled) $u_n$ along which $u_n\geq\overline\delta$ on a number of edges that diverges as $n\to\infty$ (since $x_n\to\infty$), in turn yielding $\|u_n\|_2\to\infty$ as $n\to\infty$ and contradicting the boundedness in $H^1(\G)$. Since the same argument shows that \eqref{eq:delta} remains true when replacing $e_{0,k}$ by $e_{1,k}$, it follows that for every $\delta>0$ there exists $N_\delta\in\N$ such that, for every $n\geq N_\delta$, there is at least one value $k_{\delta,n}\in\Z\cap[0,x_n]$ for which
	\[
	\max\left\{\min_{e_{0,k_{\delta,n}}}u_n,\min_{e_{1,k_{\delta,n}}}u_n\right\}\leq\delta\,.
	\]
	Consider then the following construction. For fixed $n$, take one point in $e_{0,k_{\delta,n}}$ where $u_n$ realizes the value $\displaystyle m_{n,0}:=\min_{e_{0,k_{\delta,n}}}u_n$, one point in $e_{1,k_{\delta,n}}$ where $u_n$ realizes $\displaystyle m_{n,1}:=\min_{e_{1,k_{\delta,n}}}u_n$, and split $\G$ in two disjoint parts $\G_{n,1},\G_{n,2}$ by cutting $\displaystyle e_{0,k_{\delta,n}}, e_{1,k_{\delta,n}}$ at these points. Denote by $u_{n,1},u_{n,2}$ the restrictions of $u_n$ to $\G_{n,1},\G_{n,2}$ respectively. It is straightforward to see that we can construct (e.g. by extending $u_n$ linearly to zero at the cutting points on additional intervals of fixed length) two new functions $\widetilde{u}_{n,1},\widetilde{u}_{n,2}\in H^1(\G)$ such that, for $i=1,2$,
	\begin{equation}
		\label{eq:utilde}
\|\widetilde{u}_{n,i}'\|_{L^2(\G)}^2=\|u_{n,i}'\|_{L^2(\G_{n,i})}^2+ C \alpha_{2,n}\,,\qquad	\|\widetilde{u}_{n,i}\|_{L^r(\G)}^r=\|u_{n,i}\|_{L^r(\G_{n,i})}^r+ C_r \alpha_{r,n},\quad\forall r\geq2,
	\end{equation}
	where 
	\[
	\alpha_{r,n}:=m_{n,0}^r+m_{n,1}^r\,,\qquad\forall r\geq2,
	\]
	and
	with $C_r,C>0$ independent of $n$ and $\delta$, and uniformly bounded from above and away from zero on $r\in[2,2\overline{p}]$. By the boundedness in $H^1(\G)$ of $(u_n)_n$ it then follows that $(\widetilde{u}_{n,i})_n$ is bounded in $H^1(\G)$ too, and since by construction $\|\widetilde{u}_{n,i}\|_\infty\geq1$ (because both $u_{n,1}$ and $u_{n,2}$ are not smaller than 1 at least at one point),  we also have that $\|\widetilde{u}_{n,i}\|_{2}$ (and thus also $\|u_{n,i}\|_2$) is bounded away from 0 uniformly on $n$ (e.g. by the standard Gagliardo-Nireberg inequality $\|v\|_\infty\leq\|v'\|_2^{1/2}\|v\|_2^{1/2}$, that holds for every $v\in H^1(\G)$). Hence, $\widetilde{u}_{n,i}\in\NN_{p_n,\lambda_{n,i}}(\G)$, for some $\lambda_{n,i}\in\R$  bounded from above uniformly on $n$ and such that
	\[
	\begin{split}
		\sum_{i=1}^2\|u_{n,i}\|_{L^2(\G_{n,i})}^2&\,=\|u_n\|_{L^2(\G)}^2=\|u_n\|_{L^{p_n}(\G)}^{p_n}-\|u_n'\|_{L^2(\G)}^2=\sum_{i=1}^2\|u_{n,i}\|_{L^{p_n}(\G_{n,i})}^{p_n}-\|u_n'\|_{L^2(\G)}^2\\
		&\,=\sum_{i=1}^2\|\widetilde{u}_{n,i}\|_{L^{p_n}(\G)}^{p_n}-2C_{p_n}\alpha_{p_n,n}-\|u_n'\|_{L^2(\G)}^2\\
		&\,=\sum_{i=1}^2\left(\|\widetilde{u}_{n,i}'\|_{L^2(\G)}^2+\lambda_{n,i}\|\widetilde{u}_{n,i}\|_{L^2(\G)}^2\right)-2C_{p_n}\alpha_{p_n,n}-\|u_n'\|_{L^2(\G)}^2\\
		&\,=\sum_{i=1}^2\lambda_{n,i}\|u_{n,i}\|_{L^2(\G_{n,i})}^2+(2C+C_2(\lambda_{n,1}+\lambda_{n,2}))\alpha_{2,n}-2C_{p_n}\alpha_{p_n,n}\,.
	\end{split}
	\]
	Rearranging terms, this is
	\[
	\lambda_{n,1}=1+(1-\lambda_{n,2})\frac{\|u_{n,2}\|_2^2}{\|u_{n,1}\|_2^2}-\frac{2C+C_2(\lambda_{n,1}+\lambda_{n,2})}{\|u_{n,1}\|_2^2}\alpha_{2,n}+\frac{2C_{p_n}}{\|u_{n,1}\|_2^2}\alpha_{p_n,n}\,.
	\]
	Observe that, since by construction $0<\alpha_{p_n,n}\leq2\delta^{p_n}$, taking a sequence $\delta_n\to0$ as $n\to\infty$ and a suitable subsequence (not relabeled) along which $\displaystyle \overline{\lambda}_i:=\lim_{n\to\infty}\lambda_{n,i}$ exist for both $i=1,2$ (this can be done since $(\lambda_{n,i})_n$ is bounded uniformly on $n$), passing to the limit as $n\to\infty$ we obtain
	\[
	\overline{\lambda}_1=1+\theta(1-\overline{\lambda}_2)
	\]
	for some $\theta>0$ because $\|u_{n,i}\|_2$ is bounded from above and away from zero uniformly on $n$, for both $i=1,2$.  In turn, this implies either that the largest between $\overline{\lambda}_1$ and $\overline{\lambda}_2$, say $\overline{\lambda}_1$, is strictly greater than 1, or that $\overline{\lambda}_1=\overline{\lambda}_2=1$. Hence, for large $n$, either $\lambda_{n,1}>1$ or both $\lambda_{n,1},\lambda_{n,2}$ tend to 1. But this is impossible, because by \eqref{eq:utilde} and the definition of $\lambda_{n,i}$ we have
	\[
	\begin{split}
	\JJ_{p_n}(1)&\,=\left(\frac12-\frac1{p_n}\right)\|u_n\|_{L^{p_n}(\G)}^{p_n}=\left(\frac12-\frac1{p_n}\right)\left(\|\widetilde{u}_{n,1}\|_{L^{p_n}(\G)}^{p_n}+\|\widetilde{u}_{n,2}\|_{L^{p_n}(\G)}^{p_n}-2C_{p_n}\alpha_{p_n,n}\right)\\
	&\,\geq \JJ_{p_n}(\lambda_{n,1})+\JJ_{p_n}(\lambda_{n,2})-2C_{p_n}\left(\frac12-\frac1{p_n}\right)\delta_n^{p_n}>\JJ_{p_n}(1)\,,
	\end{split}
	\]
	the last inequality following by the monotonicity and the non-negativity of $\JJ_{p_n}$ and the fact that, for sufficiently large $n$, $\JJ_{p_n}(\lambda_{n,1})+\JJ_{p_n}(\lambda_{n,2})$ is either greater than $\JJ_{p_n}(\overline{\lambda}_1)-\delta_n>\JJ_{p_n}(1)$ (if $\overline{\lambda}_1>1$), or greater than $2\JJ_{p_n}(1)-\delta_n>\JJ_{p_n}(1)$ (if $\overline{\lambda}_1=\overline{\lambda}_2=1$). This provides the contradiction we seek and completes the proof of \eqref{eq:compactA}.
	
	\smallskip
	{\em Step 2.} In view of Step 1, there exists $\overline{k}\in\Z$, independent of $n$, such that $u_n<1$ on $e_{i,k}$, for every $i=0,1,2$ and $|k|\geq\overline k$. Consider now, for every $n$, the function $f_n:[\overline k,\infty)\cap \Z\to\R$ defined by
	\[
	f_n(j):=\max\left\{u_n(\vv_{j,0}), u_n(\vv_{j,1})\right\}\,,
	\]
	where $\vv_{j,0},\vv_{j,1}\in\V_\G$ are the vertices of the edge $e_{2,j}$ with coordinates $(j,0)$, $(j,1)$ respectively.  
	
	We claim that $f_n$ is non-increasing on $[\overline k,\infty)\cap \Z$ for every $n$. 
	Indeed, assume by contradiction that this is not the case, namely that there exists $k\in[\overline k,\infty)\cap \Z$ such that $f_n(k)<f_n(k+1)$. With no loss of generality, let $\vv_{k+1,1}$ be such that $u_n(\vv_{k+1,1})=f_n(k+1)$.  Observe that, since $u_n$ is strictly smaller than 1 on $e_{i,j}$ for every $i=0,1,2$ and $j\geq\overline k$, the maximum of $u_n$ on each of this edge is attained at one of its vertices (because if $u_n$ has a maximum point in the interior of an edge, then by \eqref{eq:un>1} it cannot be smaller than 1 everywhere inside it). Therefore, since $u_n(\vv_{k+1,1})>\max\left\{u_n(\vv_{k,0}),u_n(\vv_{k,1})\right\}$ by assumption, the outgoing derivative of $u_n$ at $\vv_{k+1,1}$ along $e_{1,k}$ is negative. Analogously, since $u_n(\vv_{k+1,1})\geq u_n(\vv_{k+1,0})$, the same is true for the outgoing derivative of $u_n$ at $\vv_{k+1,1}$ along $e_{2,k+1}$.  By Kirchhoff conditions, this means that the outgoing derivative of $u_n$ at $\vv_{k+1,1}$ along $e_{1,k+1}$ is positive, i.e. that there exists a point on $e_{1,k+1}$ where $u_n$ is strictly greater than $u_n(\vv_{k+1,1})$. But this is impossible, since it would imply that the restriction of $u_n$ to the portion of $\G$ given by $\displaystyle \bigcup_{j\geq k+1}(e_{0,j}\cup e_{1,j}\cup e_{2,j})$ has a local maximum point somewhere outside $e_{2,k+1}$ (because this subset of $\G$ is connected, $u_n$ is continuous, there is at least one point where $u_n$ is strictly greater than what it is on $e_{2,k+1}$, and $u_n$ goes to 0 on $e_{i,j}$ as $j\to\infty$ for every $i=0,1,2$), where it would thus be greater than or equal to 1, contradicting Step 1. This proves the monotonicity of $f_n$ for every $n$.
	
 	For every $j\geq\overline k$, let then $\vv_j$ be a vertex of $e_{2,j}$ for which $u_n(\vv_j)=f_n(j)$. Since
 	\[
 	\begin{split}
 	\left|\int_{\bigcup_{i=0}^2 e_{i,j}}|u_n(y)|^2\,dy-3|u_n(\vv_j)|^2\right|&\,=\left|\int_{\bigcup_{i=0}^2 e_{i,j}}(|u_n(y)|^2-|u_n(\vv_j)|^2)\,dy\right|\\
 	&\,=\left|\int_{\bigcup_{i=0}^2 e_{i,j}}\int_{\vv_j}^y (u_n^2)'(s)\,ds\,dy\right|\leq 3\int_{\bigcup_{i=0}^2 e_{i,j}}|u_n|^2+|u_n'|^2\,dx\,,
 \end{split}
 	\]
	summing over $j$ and recalling that $(u_n)_n$ is bounded in $H^1(\G)$ it follows that $\displaystyle\sum_{j\geq\overline k}u_n(\vv_j)^2<\infty$. By the monotonicity of $f_n$ we then obtain, for every fixed $k>\overline k$,
	\[
	\sum_{j\geq\overline k}u_n(\vv_j)^2>\sum_{j=\overline k}^k u_n(\vv_j)^2\geq (k-\overline k)u_n(\vv_k)^2\,.
	\]
	Since the same argument can be repeated identically for $j\leq-\overline k$, it follows that there exists a constant $C>0$, independent of $n$, such that
	\begin{equation}
		\label{eq:decay}
		u_n(\vv_j)\leq \frac C{\sqrt{|j|-\overline k}}\qquad\forall |j|>\overline k,\quad\forall n.
	\end{equation}

	\smallskip
	{\em Step 3.} Let now $p_n\to\overline p$ as $n\to\infty$ be such that $\displaystyle\lim_{n\to\infty}\JJ_{p_n}=\liminf_{p\to\overline{p}}\JJ_p$, and $u_n\in\NN_{p_n}$ be such that $J_{p_n}(u_n)=\JJ_{p_n}$ for every $n$. As above, let $u_n>0$ on $\G$ and attain its $L^\infty$ norm somewhere in $\K_0$. We prove that, up to subsequences, 
	\begin{equation}
		\label{eq:rapporto}
		\lim_{n\to\infty}\frac{\|u_n\|_{p_n}^{p_n}}{\|u_n\|_{\overline p}^{\overline p}}=1\,.
	\end{equation}
	Observe first that, since $(u_n)_n$ is bounded in $H^1(\G)$, up to subsequences it converges weakly in $H^1(\G)$ to some $\overline u$. Note that $\overline{u}\not\equiv 0$ on $\G$, because if this were the case, since $u_n$ always attains its $L^\infty$ norm in the fixed compact set $\K_0$, the convergence in $L_{\text{\normalfont loc}}^\infty(\G)$ of $u_n$ to $\overline u$ would yield $\|u_n\|_\infty\to0$, contradicting \eqref{eq:un>1}. Hence, to obtain \eqref{eq:rapporto} we can equivalently prove that $u_n\to\overline u$ in $L^{\overline p}(\G)$ and 
	\begin{equation}
		\label{eq:diff}
		\|u_n\|_{p_n}^{p_n}-\|u_n\|_{\overline p}^{\overline p}\to0\qquad\text{as }n\to\infty.
	\end{equation}
	To do this, let $\overline k$ be as in Step 2 and set $\displaystyle \K:=\bigcup_{|j|\leq\overline k}(e_{0,j}\cup e_{1,j}\cup e_{2,j})$. Since $\K$ is compact, $u_n\to \overline{u}$ in $L^{\overline{p}}(\K)$ and $\|u_n\|_{L^{p_n}(\K)}^{p_n}-\|u_n\|_{L^{\overline p}(\K)}^{\overline p}\to0$ as $n\to\infty$, and we are left to prove the same convergence on $\G\setminus\K$. Note that by construction $\displaystyle \G\setminus\K$ is the disjoint union of the subgraphs $\displaystyle\G_-:=\bigcup_{j\leq-(\overline k +1)}(e_{0,j}\cup e_{1,j}\cup e_{2,j})$ and $\displaystyle \G_+:=\bigcup_{j\geq\overline k +1}(e_{0,j}\cup e_{1,j}\cup e_{2,j})$. We now show that the convergence we seek holds on $\G_+$, the argument is identical for $\G_-$. 
	
	By \eqref{eq:decay}, recall that $\displaystyle \|u_n\|_{L^\infty(\bigcup_{i=0}^2e_{i,j})}\leq f_n(j)<\frac C{\sqrt{j-\overline k}}$ for every $j>\overline k$. Consider then the function $g:\G_+\to\R$ given by $\displaystyle g(x):=\frac 1{\sqrt{j-\overline k}}$ for every $\displaystyle x\in\bigcup_{i=0}^2e_{i,j}$ and $j>\overline{k}$, that belongs to $L^\gamma(\G_+)$ for every $\gamma>2$. Since $|u_n(x)|^{p_n}-|u_n(x)|^{\overline{p}}\to0$ for a.e. $x\in\G_+$ as $n\to\infty$, and (recalling that $u_n<1$ on $\G_+$ and $p_n\to\overline p>2$)
	\[
	\left||u_n(x)|^{p_n}-|u_n(x)|^{\overline{p}}\right|\leq f_n(j)^{p_n}+f_n(j)^{\overline p}\leq Cg(x)^\gamma\,\qquad\forall x\in\bigcup_{i=0}^2e_{i,j},\,\forall j>\overline k\,,
	\]
	for every $n$ large enough and suitable $C>0$ and $\gamma>2$ independent of $n$, the Dominated Convergence Theorem gives \eqref{eq:diff} on $\G_+$. The strong convergence of $u_n$ to $\overline u$ in $L^{\overline{p}}(\G)$ can be obtained in the same way. This completes the proof of \eqref{eq:diff}, and thus that of \eqref{eq:rapporto}.
	
	Since \eqref{eq:rapporto} is equivalent to $\displaystyle\lim_{n\to\infty}\sigma_{\overline p}(u_n)=1$ (up to subsequences), arguing exactly as in \eqref{eq:convJpnp} shows that, up to subsequences,
	\[
	\JJ_{\overline{p}}\leq \lim_{n\to\infty}\JJ_{p_n}
	\]
	and the inequality is actually true for the whole sequence $(u_n)_n$ since the limit on the right side exists by assumption and is equal to $\displaystyle \liminf_{p\to\overline{p}}\JJ_p$. This yields \eqref{eq:liminf} and concludes the proof of the lemma. 
	\end{proof}
Exploiting Lemma \ref{lem:contp} we can now retrieve \eqref{eq:boundM} on ladder graphs. Indeed, if for fixed $\lambda>0$ we take $p_n\to\overline p$ and $u_n\in\NN_{p_n,\lambda}(\G)$ such that $\JJ_{p_n,\G}(\lambda)=J_{p_n,\lambda}(u_n,\G)$, arguing as in the proof of Lemma \ref{lem:contp} we obtain that up to subsequences $u_n\rightharpoonup u$ in $H^1(\G)$, $u_n\to u$ in $L^{\overline{p}}(\G)$, and $\|u_n\|_{p_n}^{p_n}\to\|u\|_{\overline{p}}^{\overline{p}}$ as $n\to\infty$. By weak lower semicontinuity, this implies $\sigma_{\overline{p},\lambda}(u)\leq 1$. The usual argument, together with Lemma \ref{lem:contp}, then shows that $u\in\NN_{\overline{p},\lambda}(\G)$ and $\JJ_{\overline{p},\G}(\lambda)=J_{\overline{p},\lambda}(u,\G)$. Furthermore, since $u_n\rightharpoonup u$ in $H^1(\G)$, $u\in\NN_{\overline{p},\lambda}(\G)$, $u_n\in\NN_{p_n,\lambda}(\G)$ and $\|u_n\|_{p_n}^{p_n}\to\|u\|_{\overline{p}}^{\overline{p}}$, the convergence of $u_n$ to $u$ is strong in $H^1(\G)$, and thus in $L^2(\G)$, yielding \eqref{eq:boundM}.

The extension of Proposition \ref{prop:est6} (that reduces to \eqref{eq:critmasR} and \eqref{eq:JGJR} since the ladder has no pendant) does not require any comment (its proof could be even simplified a bit on the ladder), and the same is true for Lemma \ref{lem:l1>l2}. Hence, to complete the proof of Theorem \ref{thm:per} we are left again to show that there exist $0<\lambda_1<\lambda_2$ such that $M_{6,\G}^-(\lambda_1)>M_{6,\G}^+(\lambda_2)$. This is done by combining Proposition \ref{prop:est6} with the next lemma.
\begin{lemma}
	\label{lem:GtoR}
	For every $\varepsilon>0$ there exists $\widehat{\lambda}>0$ (depending on $\varepsilon$) such that
	\begin{equation}
		\label{eq:muLto0}
		\sqrt{6}\,\mu_\R-\varepsilon\leq M_{6,\G}^-(\lambda)\leq M_{6,\G}^+(\lambda)\leq \sqrt{6}\,\mu_\R+\varepsilon\qquad\forall\lambda\leq\widehat{\lambda}.
	\end{equation}
\end{lemma}
\begin{proof}
	Observe as usual that, if $u\in\NN_{6,\lambda}(\G)$ is such that $J_{6,\lambda}(u,\G)=\JJ_{6,\G}(\lambda)$, then $u_\lambda(x):=\lambda^{-1/4}u(x/\sqrt{\lambda})\in\NN_{6,1}(\G_\lambda)$ satisfies $J_{6,1}(u_\lambda,\G_\lambda)=\JJ_{6,\G_{{\lambda}}}(1)$ and $\|u_\lambda\|_{L^2(\G_\lambda)}=\|u\|_{L^2(\G)}$, with $\G_\lambda:=\sqrt{\lambda}\,\G$. Hence, \eqref{eq:muLto0} is proved if we show that
	\begin{equation}
		\label{eq:Lto0}
		\lim_{\lambda\to0^+}M_{6,\G_\lambda}^\pm(1)=\sqrt{6}\,\mu_\R\,.
	\end{equation}
To prove \eqref{eq:Lto0}, we will often think of $\G_\lambda$ as embedded in $\R^2$, denoting by 
\[
\begin{split}
&E_0^\lambda:=\bigcup_{k\in\Z}\left[\sqrt{\lambda}\,k,\sqrt{\lambda}(k+1)\right]\times\{0\}, \qquad E_1^\lambda:=\bigcup_{k\in\Z}\left[\sqrt{\lambda}\,k,\sqrt{\lambda}(k+1)\right]\times\left\{\sqrt{\lambda}\right\},\\ &\qquad\qquad\qquad\qquad\qquad\qquad E_2^\lambda:=\bigcup_{k\in\Z}\left\{\sqrt{\lambda}\,k\right\}\times\left[0,\sqrt{\lambda}\right],
\end{split}
\]
the analogues of $E_0,E_1,E_2$ as in \eqref{eq:GinR2}, and by $e_{i,k}^\lambda\in E_i^\lambda, i=0,1,2$, the corresponding edges.

In what follows, a key role will be played by the variational problem on $\R$
\[
\widetilde{\JJ}_\R:=\inf_{v\in\widetilde{\NN}(\R)}\widetilde{J}(v,\R)\,,
\]
where
\[
\widetilde{J}(v,\R):=\|v'\|_{L^2(\R)}^2+\frac32\|v\|_{L^2(\R)}^2-\frac12\|v\|_{L^6(\R)}^6
\]
and
\[
\begin{split}
\widetilde{\NN}(\R):=&\,\left\{v\in H^1(\R)\,:\,\widetilde{J}'(v,\R)v=0\right\}\\
=&\,\left\{v\in H^1(\R)\,:\,2\|v'\|_{L^2(\R)}^2+3\|v\|_{L^2(\R)}^2=3\|v\|_{L^6(\R)}^6\right\}\,.
\end{split}
\]
It is easy to check that (up to translations and changes of sign) the unique solution of $\widetilde{\JJ}_\R$ is given by 
\begin{equation}
	\label{eq:phitilde}
	\widetilde{\phi}(x):=\phi_1\left(\sqrt{\frac32}\,x\right),
\end{equation} 
where $\phi_1$ is the soliton \eqref{eq:phi1}.

The rest of the proof is divided in four steps.

\smallskip
{\em Step 1.} We first show that 
\begin{equation}
	\label{eq:J<Jtilde}
	\limsup_{\lambda\to0^+}\JJ_{6,\G_{{\lambda}}}(1)\leq\widetilde{\JJ}_\R\,.
\end{equation} 
To this end, let $v\in\widetilde{\NN}(\R)$ be fixed, and consider the function $v_\lambda:\G_\lambda\to\R$ defined by
\[
v_\lambda(x):=\begin{cases}
	v(x) & \text{if }x\in E_0^\lambda\\
	v(x) & \text{if }x\in E_1^\lambda\\
	v\big(\sqrt{\lambda}\,k\big) & \text{if }x\in e_{2,k}^\lambda,\text{ for some }k\in\Z\,,
\end{cases}
\] 
where with a slight abuse of notation we identified $E_0^\lambda, E_1^\lambda$ with two real lines. Since $v\in H^1(\R)$, by definition for every $r\geq2$ 
\[
\|v_\lambda\|_{L^r(E_2^\lambda)}^r=\sqrt{\lambda}\sum_{k\in\Z}\left|v\big(\sqrt{\lambda}\,k\big)\right|^r\to\|v\|_{L^r(\R)}^r\qquad\text{as }\lambda\to0^+\,,
\]
so that, for $\lambda\to0^+$,
\[
\sigma_{6,1}(v_\lambda)=\left(\frac{\|v_\lambda'\|_{L^2(\G_\lambda)}^2+\|v_\lambda\|_{L^2(\G_\lambda)}^2}{\|v_\lambda\|_{L^6(\G_\lambda)}^6}\right)^\frac14=\left(\frac{2\|v'\|_{L^2(\R)}^2+3\|v\|_{L^2(\R)}^2+o(1)}{3\|v\|_{L^6(\R)}^6+o(1)}\right)^\frac14=1+o(1)\,,
\]
and thus
\[
\JJ_{6,\G_\lambda}(1)\leq J_{6,1}(\sigma_{6,1}(v_\lambda)v_\lambda,\G_\lambda)=\frac13\sigma_{6,1}(v_\lambda)^6\|v_\lambda\|_{L^6(\G_\lambda)}^6=\|v\|_{L^6(\R)}^6+o(1)=\widetilde{J}(v,\R)+o(1)
\]
(the last equality follows by the definition of $\widetilde{J}(\cdot,\R)$ and $\widetilde{\NN}(\R)$). Passing to the limsup as $\lambda\to0^+$ and taking the infimum over $v\in\widetilde{\NN}(\R)$ gives \eqref{eq:J<Jtilde}.

\smallskip
{\em Step 2. } Let now $u_\lambda\in\NN_{6,1}(\G_\lambda)$ be such that $J_{6,1}(u_\lambda,\G_\lambda)=\JJ_{6,\G_{{\lambda}}}(1)$. As usual, assume with no loss of generality $u_\lambda>0$ on $\G_\lambda$ and that, by the periodicity of the graph, it always attains its $L^\infty$ norm somewhere in the set $\K_0^\lambda:=\displaystyle\bigcup_{i=0}^2 e_{i,0}^\lambda$. 

For $i=0,1,2$, denote by $u_{\lambda,i}$ the restriction of $u_\lambda$ to the set $E_i^\lambda$. Note that $u_{\lambda,0}$ and $u_{\lambda,1}$ can be thought of as functions in $H^1(\R)$, whereas $u_{\lambda,2}$ can be seen as a function in $L^r(\R)$, for every $r\geq2$. We then show that
\begin{equation}
	\label{eq:limLr}
	\left\|u_{\lambda,i}-u_{\lambda,j}\right\|_{L^r(\R)}=o(1)\qquad\text{as }\lambda\to0^+
\end{equation}
for every $r\in[2,\infty)$ and every $i,j=0,1,2$, $i\neq j$.

Let us prove \eqref{eq:limLr} explicitly when $i=0$, $j=1$ (the other cases are analogous). To this end, for every $k\in\Z$ and $x\in\big[0,\sqrt{\lambda}\big]$, with a slight abuse of notation we denote by $u_{\lambda,0}(x), u_{\lambda,1}(x)$ the values of $u_\lambda$ at the points of $e_{0,k}^\lambda, e_{1,k}^\lambda$ corresponding to $(x,0), (x,\sqrt{\lambda})$, respectively. Then , by H\"older inequality,
\[
\left|u_{\lambda,0}(x)-u_{\lambda,1}(x)\right|^2\leq3\sqrt\lambda \int_{\bigcup_{i=0}^2e_{i,k}^\lambda}|u_\lambda'|^2\,dy\qquad\forall x\in\big[0,\sqrt{\lambda}\big],\forall k\in\Z,
\]
so that, integrating in $x$ over $\big[0,\sqrt{\lambda}\big]$ and summing over $k\in\Z$,
\[
\|u_{\lambda,0}-u_{\lambda,1}\|_{L^2(\R)}^2\leq 3 \lambda\sum_{k\in\Z}\int_{\bigcup_{i=0}^2e_{i,k}^\lambda}|u_\lambda'|^2\,dy=3\lambda\|u_\lambda'\|_{L^2(\G_\lambda)}^2\,,
\]
in turn yielding \eqref{eq:limLr} with $r=2$ since, by \eqref{eq:J<Jtilde}, $\|u_\lambda\|_{H^1(\G_\lambda)}$ is bounded from above uniformly on $\lambda\to0^+$. The inequality for $r>2$ then follows by the one with $r=2$ and the uniform boundedness of $u_\lambda$ in $L^\infty(\G_\lambda)$.

\smallskip
{\em Step 3.} We now prove that 
\begin{equation}
	\label{eq:J>Jtilde}
	\widetilde{\JJ}_\R\leq\liminf_{\lambda\to0^+}\JJ_{6,\G_\lambda}(1).
\end{equation}
To this end, recall first that, for any given point of a ladder graph, one can always construct two disjoint paths in the graph of infinite length starting at that point. This ensures that, for every $\lambda$, any non-negative function in $H^1(\G_\lambda)$ attains almost all the values in its image at least twice on the graph. As a consequence, standard rearrangement techniques (see e.g. \cite[Section 2]{DDGS}) imply that $\JJ_{6,\G_{{\lambda}}}(1)\geq\JJ_{6,\R}(1)$ for every $\lambda$, in turn guaranteeing that $\displaystyle \liminf_{\lambda\to0^+}\JJ_{6,\G_\lambda}(1)$ is bounded away from 0.

Let then $u_\lambda\in\NN_{6,1}(\G_\lambda)$ be such that $J_{6,1}(u_\lambda,\G_\lambda)=\JJ_{6,\G_{{\lambda}}}(1)$ for every $\lambda$, and $\displaystyle\lim_{\lambda\to0^+}J_{6,1}(u_\lambda,\G_\lambda)=\liminf_{\lambda\to0^+}\JJ_{6,\G_\lambda}(1)$. By Steps 1--2 and the fact that $J_{6,1}(u_\lambda,\G_\lambda)$ is bounded away from 0 uniformly on $\lambda$, we obtain that $\|u_{\lambda,i}\|_{L^6(\R)}$ is bounded away from zero too, for every $i=0,1,2$. Suppose with no loss of generality 
\begin{equation}
\label{eq:u'0<u'1}	
\|u_{\lambda,0}'\|_{L^2(\R)}\leq\|u_{\lambda,1}'\|_{L^2(\R)}\qquad\forall\lambda\,,
\end{equation}
so that, by \eqref{eq:limLr} and $u_\lambda\in\NN_{6,1}(\G_\lambda)$,
\[
\begin{split}
2\|u_{\lambda,0}'\|_{L^2(\R)}^2+3\|u_{\lambda,0}\|_{L^2(\R)}^2&\,\leq\sum_{i=0}^2\left(\|u_{\lambda,i}'\|_{L^2(\R)}^2+\|u_{\lambda,i}\|_{L^2(\R)}^2\right)+o(1)\\
&\,=\|u_{\lambda}'\|_{L^2(\G_\lambda)}^2+\|u_{\lambda}\|_{L^2(\G_\lambda)}^2+o(1)=\|u_\lambda\|_{L^6(\G_\lambda)}^6+o(1)\\
&\,=\sum_{i=0}^2\|u_{\lambda,i}\|_{L^6(\R)}^6+o(1)=3\|u_{\lambda,0}\|_{L^6(\R)}^6+o(1)\qquad\text{as $\lambda\to0^+$,}
\end{split}
\]
that is 
\[
\widetilde{\sigma}(u_{\lambda,0}):=\left(\frac{2\|u_{\lambda,0}'\|_{L^2(\R)}^2+3\|u_{\lambda,0}\|_{L^2(\R)}^2}{3\|u_{\lambda,0}\|_{L^6(\R)}^6}\right)^\frac14\leq 1+o(1)\qquad\text{as }\lambda\to0^+\,.
\]
Since $\widetilde{\sigma}(u_{\lambda,0})u_{\lambda,0}\in\widetilde{\NN}(\R)$ for every $\lambda$, as $\lambda\to0^+$ we then have
\[
\widetilde{\JJ}_\R\leq \widetilde{J}(\widetilde{\sigma}(u_{\lambda,0})u_{\lambda,0},\R)=\widetilde{\sigma}(u_{\lambda,0})^6\|u_{\lambda,0}\|_{L^6(\R)}^6\leq\frac13\|u_\lambda\|_{L^6(\G_\lambda)}^6+o(1)= J_{6,1}(u_\lambda,\G_\lambda)+o(1)\,,
\]
i.e. \eqref{eq:J>Jtilde}. 

\smallskip
{\em Step 4. } Combining \eqref{eq:J<Jtilde} and \eqref{eq:J>Jtilde} entails $\displaystyle\lim_{\lambda\to0^+}\JJ_{6,\G_{{\lambda}}}(1)=\widetilde{\JJ}_\R$. Moreover, the argument in Step 3 also implies
$\widetilde{\sigma}(u_{\lambda,0})\to1$ and $\widetilde{J}(\widetilde{\sigma}(u_{\lambda,0})u_{\lambda,0},\R)\to\widetilde{\JJ}_\R$  as $\lambda\to0^+$. All in all, this is enough to conclude that $u_{\lambda,0}$ converges strongly in $H^1(\R)$ to the function $\widetilde{\phi}$ in \eqref{eq:phitilde}, so that
\[
\lim_{\lambda\to0^+}\|u_\lambda\|_{L^2(\G_\lambda)}^2=3\lim_{\lambda\to0^+}\|u_{\lambda,0}\|_{L^2(\R)}^2=3\|\widetilde{\phi}\|_{L^2(\R)}^2=3\sqrt{\frac23}\|\phi_1\|_{L^2(\R)}^2=\sqrt{6}\,\mu_\R\,.
\]
Since the limit holds along any sequence of ground states of $J_{6,1}(\cdot,\G_\lambda)$ in $\NN_{6,1}(\G_\lambda)$, it gives \eqref{eq:Lto0} and completes the proof of the lemma.
\end{proof}

\section*{Acknowledgements}
This study was carried out within the project E53D23005450006 “Nonlinear dispersive equations in presence of singularities” – funded by European Union – Next Generation EU within the PRIN 2022 program (D.D. 104 -02/02/2022 Ministero dell’Università e della Ricerca). This manuscript reflects only the author’s views and opinions and the Ministry cannot be considered responsible for them.

\end{document}